\documentclass[times]{amsart}
\usepackage{amsfonts,amssymb,graphicx,enumerate,natbib,youngtab}

\numberwithin{equation}{section}

\theoremstyle{definition}
\newtheorem{remark}{Remark}

\theoremstyle{plain}
\newtheorem{theorem}{Theorem}[section]
\newtheorem{proposition}[theorem]{Proposition}
\newtheorem{corollary}[theorem]{Corollary}
\newtheorem{lemma}[theorem]{Lemma}

\newcommand{\CC}{\mathbb{C}}
\newcommand{\HH}{\mathbb{H}}

\newcommand{\ZZ}{\mathbb{Z}}

\newcommand{\hh}{\mathfrak{h}}

\newcommand{\la}{\langle}
\newcommand{\ra}{\rangle}
\newcommand{\one}{\mathbf{1}}

\newcommand{\ttt}{\mathfrak{t}}

\newcommand{\bone}{b_1}
\newcommand{\btwo}{b_2}
\newcommand{\bthree}{b_3}
\newcommand{\bfour}{b_4}
\newcommand{\bfive}{b_5}
\newcommand{\bsix}{b_6}
\newcommand{\bseven}{b_7}

\begin{document}

\begin{abstract}
The rational Cherednik algebra $\HH$ is a certain algebra of differential-reflection operators attached to a complex reflection group $W$ and depending on a set of central parameters.  Each irreducible representation $S^\lambda$ of $W$ corresponds to a standard module $M(\lambda)$ for $\HH$.  This paper deals with the infinite family $G(r,1,n)$ of complex reflection groups; our goal is to study the standard modules using a commutative subalgebra $\ttt$ of $\HH$ discovered by Dunkl and Opdam.  In this case, the irreducible $W$-modules are indexed by certain sequences $\lambda$ of partitions.  We first show that $\ttt$ acts in an upper triangular fashion on each standard module $M(\lambda)$, with eigenvalues determined by the combinatorics of the set of standard tableaux on $\lambda$.  As a consequence, we construct a basis for $M(\lambda)$ consisting of orthogonal functions on $\CC^n$ with values in the representation $S^\lambda$.  For $G(1,1,n)$ with $\lambda=(n)$ these functions are the non-symmetric Jack polynomials.  We use intertwining operators to deduce a norm formula for our orthogonal functions and give an explicit combinatorial description of the lattice of submodules of $M(\lambda)$ in the case in which the orthogonal functions are all well-defined.  A consequence of our results is the construction of a number of interesting finite dimensional modules with intricate structure.  Finally, we show that for a certain choice of parameters there is a cyclic group of automorphisms of $\HH$ so that the rational Cherednik algebra for $G(r,p,n)$ is the fixed subalgebra.  Our results therefore descend to the rational Cherednik algebra for $G(r,p,n)$ by Clifford theory.
\end{abstract}

\title{Orthogonal functions generalizing Jack polynomials}

\author{Stephen Griffeth }

\address{Department of Mathematics \\
University of Minnesota \\
Minneapolis, MN 55455 \\}

\maketitle

\section{Introduction.}

The rational Cherednik algebra $\HH$ is an algebra attached to a complex reflection group $W$ and depending on a set of central parameters indexed by the conjugacy classes of reflections in $W$.  The representations of the rational Cherednik algebra are closely connected with those of the (finite) Hecke algebra of $W$, and with the geometry of the quotient singularity $(\hh \oplus \hh^*)/W$, where $\hh$ is the reflection representation of $W$.

In this paper we focus on the groups $G(r,1,n)$ and use the commutative subalgebra $\ttt$ of $\HH$ introduced by Dunkl and Opdam in \cite{DuOp}, section 3, together with the technique of intertwining operators, introduced for the double affine Hecke algebras by Cherednik (see, for example, \cite{Che2} and the references therein).  Our first main result is Theorem \ref{Upper triangular}, where we show that $\ttt$ acts in an upper triangular fashion on each standard module $M(\lambda)$ with simple spectrum for generic choices of the defining parameters; the theorem implies the existence of a basis $f_{\mu,T}$ of $M(\lambda)$ consisting of functions orthogonal with respect to the contravariant form on $M(\lambda)$.  Our second main result is Theorem \ref{submodules theorem}, in which we describe in a combinatorial way the set of pairs $(\mu,T)$ indexing a basis $f_{\mu,T}$ for each submodule of $M(\lambda)$, in those cases for which the orthogonal functions $f_{\mu,T}$ are all well defined.  Theorem \ref{submodules theorem} may be helpful for understanding two of the outstanding problems in the representation theory of $\HH$: what are the finite dimensional $\HH$-modules, and (when combined with the norm formula of Theorem \ref{norm formula}) what are the graded composition multiplicities $[M(\lambda_1):L(\lambda_2)]$ with respect to the Jantzen filtration of $M(\lambda_1)$?   

One benefit of our approach is that it is elementary, requiring only linear algebra and some well-known combinatorics.  A second benefit is that it provides the most explicit information currently available on the submodule structure of the standard modules (outside of the case of dihedral groups, worked out in \cite{Chm}).  The main disadvantage of our approach is that there is no obvious way of generalizing it to the exceptional complex reflection groups; perhaps this is not surprising, since the groups $G(r,p,n)$ have the most obvious combinatorial structure of all the complex reflection groups.  A second disadvantage is that one has to work harder to obtain detailed information when the spectrum of $\ttt$ is \emph{not} simple.  Part of the analysis in this more difficult case is carried out for the symmetric group $G(1,1,n)$ and the polynomial representation $M(\text{triv})$ of $\HH$ in \cite{Dun1} and \cite{Dun2}.  We do not attempt to obtain analogous results here, but it should be noted that our Theorem \ref{submodules theorem} was inspired by a question posed at the end of the paper \cite{Dun2}.  It would be interesting to generalize the results of the papers \cite{Dun1} and \cite{Dun2} to arbitrary standard modules $M(\lambda)$.  Our paper provides a starting point.

Another approach to studying the representations of the rational Cherednik algebra is via the corresponding monodromy representations of the finite Hecke algebra of $W$ (the ``Knizhnik-Zamolodchikov functor'').  This is done, for example, in \cite{GGOR} and \cite{Chm}.  It has the advantage of applying to all complex reflection groups (though at present there are only case by case proofs of certain of the necessary theorems for finite Hecke algebras of complex reflection groups).  It is our feeling that, for the infinite family $G(r,p,n)$, the two approachs (via the KZ functor or via diagonalization and intertwining operators) should be regarded as complementary.

We will now state our results more precisely.  Let $W=G(r,1,n)$.  We will assume for the introduction that the reader is familiar with standard Young tableaux on $r$-partitions of $n$ (see section \ref{notation} for our definitions).  As a vector space and for any choice of parameters, the algebra $\HH$ is isomorphic to
\begin{equation}
S(\hh^*) \otimes \CC W \otimes S(\hh).
\end{equation}
The irreducible $\CC W$-modules are indexed by sequences
$\lambda=(\lambda^0,\lambda^1,\dots,\lambda^{r-1})$ of partitions with $n$ total
boxes, and if $S^\lambda$ is the irreducible corresponding to $\lambda$ then it
has a basis $\{v_T \ | \ T \in \text{SYT}(\lambda) \}$ indexed by standard
tableaux on $\lambda$.  The action of a certain generating set of $G(r,1,n)$ on
this basis can be made quite explicit (see Theorem \ref{CW modules theorem}).  The standard module corresponding to $\lambda$ is then
\begin{equation}
M(\lambda)=\text{Ind}_{\CC W \otimes S(\hh)}^\HH S^\lambda,
\end{equation} where $S(\hh)$ acts on $S^\lambda$ via $y.v=0$ for all $y \in \hh$ and $v \in S^\lambda$.  As a vector space
\begin{equation}
M(\lambda) \simeq S(\hh^*) \otimes S^\lambda,
\end{equation} and it follows that the elements $x^\mu \otimes w_\mu^{-1}.v_T$ for $\mu \in \ZZ_{\geq 0}^n$ and $T \in \text{SYT}(\lambda)$ are a basis of $M(\lambda)$, where $w_\mu \in S_n$ is the longest element of $S_n$ so that $\mu^-=w_\mu.\mu$ is non-decreasing. 

Now we can state a simplified version of our first main result.  The more precise statement is Theorem \ref{Upper triangular}.
\begin{theorem}
  There is a partial order $\leq$ on $\ZZ_{\geq 0}^n \times \text{SYT}(\lambda)$
  so that the algebra $\ttt$ acts in an upper-triangular fashion on the basis
  $x^\mu \otimes w_\mu^{-1}.v_T$.  For generic values of the parameters the
  $\ttt$-eigenspaces are one-dimensional, and there exists a family $f_{\mu,T}$
  of elements of $M(\lambda)$ determined by
\begin{enumerate}[(a)]
\item $f_{\mu,T}=x^\mu \otimes w_\mu^{-1}.v_T+\text{lower terms}$, and
\item $f_{\mu,T}$ is a $\ttt$-eigenvector.
\end{enumerate}
\end{theorem}

We need a bit more notation to state our second main result.  As above let $\mu^-$ be the non-decreasing rearrangement of $\mu$ and let $w_\mu \in S_n$ be the longest element with $w_\mu.\mu=\mu^-$.  Let $\Gamma=\ZZ_{\geq 0}^n \times \text{SYT}(\lambda)$.  For a box $b \in \lambda$ and a positive integer $k$, define
\begin{equation*}
\Gamma_{b,k}=\{(\mu,T) \in \Gamma \ | \ \mu^-_{T^{-1}(b)} \geq k \},
\end{equation*} and for a pair $b_1,b_2 \in \lambda$ of boxes of $\lambda$ and a positive integer $k$, define
\begin{equation*}
\Gamma_{b_1,b_2,k}=\{(\mu,T) \in \Gamma \ | \ \mu^-_{T^{-1}(b_1)}-\mu^-_{T^{-1}(b_2)} > k, \ \text{or} \ \mu^-_{T^{-1}(b_1)}-\mu^-_{T^{-1}(b_2)} =k \ \text{and} \ w_\mu^{-1} T^{-1}(b_1)<w_\mu^{-1}T^{-1}(b_2) \} 
\end{equation*}  Write
\begin{equation}
M_{b,k}=\CC \text{-span} \{f_{\mu,T} \ | \ (\mu,T) \in \Gamma_{b,k} \} \quad \text{and} \quad M_{b_1,b_2,k}=\CC \text{-span} \{ f_{\mu,T} \ | \ (\mu,T) \in \Gamma_{b_1,b_2,k} \}
\end{equation} for the $\CC$-spans of $ \Gamma_{b,k}$ and $\Gamma_{b_1,b_2,k}$.  The algebra $\HH$ depends on parameters $c_0,d_0,d_1,\dots,d_{r-1}$.  Our second main result (Theorem \ref{submodules theorem}) is
\begin{theorem}
Suppose the spectrum of $M(\lambda)$ is simple (see Lemma \ref{simple spectrum}). 
\begin{enumerate}[(a)]
\item  Let $k \in \ZZ_{>0}$, $b \in \lambda^i$.  Then $M_{b,k}$ is an $\HH$-submodule of $M(\lambda)$ exactly if $k=d_i-d_{i-k}+ \text{ct}(b) r c_0$, where ct$(b)$ is the content of the box $b$. 
\item  Let $k \in \ZZ_{>0}$, $b_1 \in \lambda^i$, $b_2 \in \lambda^j$.  Then $M_{b_1,b_2,k}$ is an $\HH$-submodule of $M(\lambda)$ exactly if $k=i-j$ mod $r$ and $k=d_{i}-d_{j}+ (\text{ct}(b_1)-\text{ct}(b_2) \pm 1) r c_0 $.  .
\item Every submodule of $M(\lambda)$ is in the lattice generated by taking intersections and sums of the submodules of types (a) and (b).
\end{enumerate}
\end{theorem}  Using the theorem it is straightforward to generate many examples of finite dimensional representations of the rational Cherednik algebra of $G(r,1,n)$ (see \ref{typeB} and \ref{typer}).  As far as we know, ``most'' of these have not appeared before.  In \cite{Suz} and for $r=1$, Suzuki classifies the irreducible modules on which $\ttt$ acts with one-dimensional eigenspaces.  In many cases $L(\lambda)$ has one-dimensional eigenspaces while $M(\lambda)$ does not, so we do not deduce his results from ours.  On the other hand, we are interested in the structure of the standard modules themselves, which does not follow directly from Suzuki's work even for $r=1$.  

In Corollary~\ref{finite dimensional}, we construct a large number of finite dimensional $\HH$-modules, most of which have not appeared before.

Section 2 fixes notation, and sections 3 and 4 of the paper are a review of standard material.  We only sketch the proofs of the assertions made in these sections, providing references where appropriate.  Section 3 describes Young's orthonormal form for the complex representations of the group $G(r,1,n)$.  In the fourth section we define the rational Cherednik algebra $\HH$ and state results we need in the rest of the paper.  In the fifth section we show that a certain commutative subalgebra $\ttt \subseteq H$ acts in an upper triangular fashion on the standard $\HH$-modules, and give explicit formulas for the eigenvalues.  For generic choices of the parameters defining $\HH$, we show that each standard module has a $\ttt$-eigenbasis consisting of analogues of the non-symmetric Jack functions.  We give explicit formulas determining the $\HH$-action on this eigenbasis.  In the sixth section we compute the norms of the $\ttt$-eigenfunctions with respect to the contravariant form.  This calculation may be used to describe the radical and the irreducible head of the standard modules in those cases for which the eigenfunctions are well-defined.  The seventh section contains our combinatorial description of the submodules of $M(\lambda)$ in the case in which the $\ttt$-eigenspaces are all at most one-dimensional.  The combinatorics developed in this section for the purposes of proving Theorem \ref{submodules theorem} should also be useful for analyzing the case in which $\ttt$ acts non-semisimply on $M(\lambda)$.  Section~\ref{submodules section} concludes with some examples illustrating how Theorem \ref{submodules theorem} may be used.  In Section~\ref{Clifford section} we show how we may deduce information for the rational Cherednik algebra of type $G(r,p,n)$ from that for $G(r,1,n)$ by using the version of Clifford theory given in the appendix of \cite{RaRa}.

\section{Notation and preliminaries} \label{notation}

Let $S_n$ be the group of permutations of the set $\{1,2,\dots,n\}$.  The notation $w_1 \leq w_2$ for $w_1,w_2 \in S_n$ refers to Bruhat order, and we write $l(w)$ for the length of an element $w \in S_n$.  Let $w_0 \in S_n$ be the \emph{longest element}, with
\begin{equation}
w_0(i)=n-i+1 \quad \hbox{for $1 \leq i \leq n$.}
\end{equation}  For a sequence $\mu \in \ZZ_{\geq 0}^n$ of $n$ non-negative integers, we write $\mu^+$ for the non-increasing (partition) rearrangement of $\mu$, and $\mu^-$ for the non-decreasing (anti-partition) rearrangement of $\mu$.  For $w \in S_n$ and $\mu \in \ZZ_{\geq 0}^n$, the formula
\begin{equation}
w.\mu=(\mu_{w^{-1}(1)},\mu_{w^{-1}(2)},\dots,\mu_{w^{-1}(n)})
\end{equation} defines a left action of $S_n$ on $\ZZ_{\geq 0}^n$: the $i$th coordinate of $w_1.(w_2.\mu)$ is $(w_2.\mu)_{w_1^{-1}(i)}=\mu_{w_2^{-1} w_1^{-1}(i)}=\mu_{(w_1 w_2)^{-1}(i)}$.  Let $w_\mu$ be the longest element of $S_n$ such that $w_\mu.\mu=\mu^-$; thus
\begin{equation} \label{wmu def}
w_\mu(i)=|\{1 \leq j < i \ | \ \mu_j < \mu_i \}|+|\{i \leq j \leq n \ | \ \mu_j \leq \mu_i \}|.
\end{equation}  Let $c=(12 \cdots n)$ be the long cycle in $S_n$, let $\epsilon_i \in \ZZ_{\geq 0}^n$ be the element with a $1$ in the $i$th position and $0$'s elsewhere, and define
\begin{equation}
\phi.\mu=c^{-1}.\mu+\epsilon_n \quad \text{and} \quad
\psi.\mu=c.\mu-\epsilon_1.
\end{equation}  For $\mu \in \ZZ_{\geq 0}^n$ one has
\begin{equation}
w_{\phi.\mu}=w_\mu c \quad \hbox{and if $s_i.\mu \neq \mu$ then} \quad  w_{s_i.\mu}=w_\mu s_i.
\end{equation}

Let $m$ be a non-negative integer.  A \emph{partition} of length $m$ is a non-increasing sequence $\lambda=(\lambda_1 \geq \lambda_2 \geq \cdots \geq \lambda_m \geq 0)$ of $m$ positive integers.  Let $r$ be a positive integer.  An \emph{$r$-partition} is a sequence $\lambda=(\lambda^0,\lambda^1,\dots,\lambda^{r-1})$ of $r$ partitions of unspecified lengths---some may be empty.  The \emph{Young diagram} of an $r$ partition $\lambda$ is the graphical representation consisting of a collection of boxes stacked in a corner: the Young diagram for the $4$-partition $\left((3,3,1),(2,1),\emptyset,(5,5,2,1) \right)$ of $23$ is 

\begin{equation*}
\left(\begin{array}{cccc} \begin{array}{c} \yng(3,3,1) \end{array}, &  \begin{array}{@{}c} \yng(2,1) \end{array}, &  \begin{array}{@{}c} \emptyset \end{array}, &  \begin{array}{@{}c} \yng(5,5,2,1) \end{array}
\end{array} \right).
\end{equation*}

  A \emph{standard young tableau} $T$ on an $r$-partition $\lambda$ of $n$ is a filling of the boxes of $\lambda$ with the integers $\{1,2,\dots,n\}$ in such a way that within each partition $\lambda^i$, the entries are strictly increasing from left to right and top to bottom.  For example,

\begin{equation*}
\young(246,39),\ \young(15,78)
\end{equation*} is a standard Young tableau on the $2$-partition $(3,2),(2,2)$.  More formally, a tableau $T$ may be thought of as a bijection from the set $\{1,2,\dots,n\}$ to the set of boxes in $\lambda$.  We write
\begin{equation}
T(i)=\hbox{the box $b$ of $\lambda$ in which $i$ appears}.
\end{equation}  Then for tableaux $S$ and $T$ on $\lambda$, the permutation $S^{-1} T$ of $\{1,2,\dots,n\}$ is a measure of the distance between $S$ and $T$.  

For an $r$-partition $\lambda$ we write
\begin{equation}
\text{SYT}(\lambda)=\{\hbox{standard Young tableaux on $\lambda$} \}.
\end{equation}  We define the \emph{content} of a box $b \in \lambda^i$ to be $j-i$ if $b$ is in the $i$th row and $j$th column of $\lambda^i$.  We write
\begin{equation} \label{content def}
\text{ct}(b)=\hbox{content of $b$.}
\end{equation}  We also define the function $\beta$ on the set of boxes of $\lambda$ by 
\begin{equation} \label{beta def}
\beta(b)=i \quad \hbox{if $b \in \lambda^i$.}
\end{equation}  Thus for the tableau $T$ pictured above, one has
\begin{equation*}
\beta(T(7))=1 \quad \text{and} \quad \text{ct}(T(3))=-1.
\end{equation*}

Fix positive integers $r$ and $n$ and a positive integer $p$ dividing $r$.  Let
\begin{equation}
W=G(r,p,n)
\end{equation} be the group of $n$ by $n$ monomial matrices whose entries are $r$th roots of $1$, and so that the product of the non-zero entries is an $r/p$th root of $1$.  We write
\begin{equation}
\CC W=\CC \text{-span} \{t_w \ | \ w \in W \} \quad \text{with multiplication} \quad t_{w_1} t_{w_2}=t_{w_1 w_2} \ \hbox{for $w_1,w_2 \in W$}
\end{equation} for the group algebra of $W$.  Let $\zeta=e^{2 \pi i/r}$ and
\begin{equation}
\zeta_i=\text{diag}(1,\dots,\zeta,\dots,1)
\end{equation} be the diagonal matrix with a $\zeta$ in the $i$th position and $1$'s elsewhere on the diagonal.  Let
\begin{equation}
s_{ij}=(ij) \quad \text{and} \quad s_i=s_{i,i+1}
\end{equation} be the transposition interchanging $i$ and $j$ and the $i$th simple transposition, respectively.

\section{Representations of $G(r,1,n)$ via Jucys-Murphy elements.}

In this section we review the approach to the complex representations of $W=G(r,1,n)$ via Jucys-Murphys elements.  This approach has been known for quite some time.  It appears that the Jucys-Murphy elements we will use appeared first in \cite{Che1} for the groups $G(2,1,n)$ and $G(2,2,n)$, and for the groups $G(r,p,n)$ in \cite{RaSh}.  The results, in the form that we will use them, may be found in unpublished notes \cite{Ram} and \cite{Gri3}.  We will use these results to give a combinatorial description of the standard modules for the rational Cherednik algebra of type $G(r,1,n)$.

For $1 \leq i \leq n$ define a \emph{Jucys-Murphy element} $\phi_i \in \CC W$ by
\begin{equation} \label{psidef}
\phi_i=\sum_{1 \leq j <i} \sum_{l=0}^{r-1} t_{\zeta_i^l s_{ij} \zeta_i^{-l}}.
\end{equation}  Let $\mathfrak{s}$ be the subalgebra of $\CC W$ generated by $\phi_i$ and $t_{\zeta_i}$ for $1 \leq i \leq n$.  We will describe the eigenspace decomposition of the irreducible $\CC W$-modules with respect to $\mathfrak{s}$ in terms of certain tableaux. 

The following theorem, stated in this way, may be found in the notes at \cite{Gri3}; see Corollary 3.5 and section 4.  Equivalent results are contained in \cite{Ram}.  Similar results are given in the papers \cite{Ari} and \cite{ArKo}.  It should be considered well-known, though we do not know of a published reference in which it is stated in the form we shall use it.  For $1 \leq i \leq n-1$ define operators $\tau_i$ on $\CC W$-modules by
\begin{equation}
\tau_i=t_{s_i}+\frac{1}{\phi_i - \phi_{i+1}} \pi_i \quad \text{where} \quad \pi_i=\sum_{l=0}^{r-1} t_{\zeta_i \zeta_{i+1}^{-1}}^l.
\end{equation}  By part (b) of the Theorem~\ref{CW modules theorem}, $\tau_i$ is well-defined on every $\CC W$ module.  Let $P_{r,n}$ denote the set of $r$-partitions of $n$.

\begin{theorem} \label{CW modules theorem}
The irreducible representations of $\CC W$ may be parametrized by $\lambda \in P_{r,n}$ in such a way that if $S^{\lambda}$ is the irreducible module corresponding to $\lambda$, then $S^{\lambda}$ has a basis $v_T$ indexed by $T \in \text{SYT}(\lambda)$ with

\begin{enumerate}[(a)]

\item  For each $T \in \text{SYT}(\lambda)$, the $\mathfrak{s}$-weight of $v_T$ is given by
\begin{equation*}
\phi_i.v_T=r \text{ct}(T(i)) v_T \quad \text{ and} \quad t_{\zeta_i}.v_T=\zeta^{\beta(T(i))} v_T.
\end{equation*} where the functions ct and $\beta$ are defined in \eqref{content def} and \eqref{beta def}.
 
\item   For each $T \in \text{SYT}(\lambda)$ and $1\leq i \leq n-1$,
\begin{equation*}
\tau_i.v_T=\begin{cases} 0 \quad & \hbox{if $s_i.T$ is not a standard tableau,} \\
v_{s_i.T} \quad &\hbox{if $\zeta^{\beta(T(i))} \neq \zeta^{\beta(T(i+1))}$, and} \\
\left(1-\left(\frac{1}{\text{ct}(T(i+1))-\text{ct}(T(i))} \right)^2 \right)^{1/2} v_{s_i.T} \quad &\hbox{if $s_i.T$ is a standard tableau and $\zeta^{\beta(T(i))}=\zeta^{\beta(T(i+1))}$.}
\end{cases}
\end{equation*} 
\end{enumerate}
\end{theorem}

Observe that the $\mathfrak{s}$-eigenspaces on $S^{\lambda}$ are one-dimensional.  In particular, the dimension of $S^{\lambda}$ is the number of standard tableaux on $\lambda$.

Let $\mu \in \ZZ_{\geq 0}^n$.  In our analysis of the spectrum of the standard modules for the rational Cherednik algebra we will need the elements $\phi_i^\mu$ defined by
\begin{equation}
\phi_i^\mu=\sum_{\substack{ 1 \leq j < i \\ \mu_j < \mu_i \\ 0 \leq l \leq r-1}} t_{\zeta_i^l s_{ij} \zeta_i^{-l}}  
 +\sum_{\substack{ i < j \leq n \\ \mu_j \leq \mu_i \\ 0 \leq l \leq r-1}} t_{\zeta_i^l s_{ij} \zeta_i^{-l}} \quad \hbox{for $1 \leq i \leq n$.}
\end{equation}  One checks by direct calculation that with $w_\mu$ as in \eqref{wmu def}
\begin{equation*}
t_{w_\mu}^{-1} \phi_{w_\mu(i)} t_{w_\mu} = \phi_i^\mu \quad \hbox{for $1 \leq i \leq n$ and $\mu \in \ZZ_{\geq 0}^n$.}
\end{equation*}  Hence if
\begin{equation} \label{JM basis}
v_T^\mu=t_{w_\mu}^{-1}.v_T
\end{equation} then
\begin{equation}
\phi_i^\mu.v_T^\mu=r \text{ct}(T(w_\mu(i))) v_T^\mu \quad \text{and} \quad t_{\zeta_i}.v_T^\mu=\zeta^{\beta(T(w_\mu(i)))} v_T^\mu.
\end{equation} 

\section{The rational Cherednik algebra of type $G(r,1,n)$.}

Let
\begin{equation*}
y_i=(0,\dots,1,\dots,0)^t \quad \text{and} \quad x_i=(0,\dots,1,\dots,0)
\end{equation*} have $1$'s in the $i$th position and $0$'s elsewhere, so that $y_1,\dots,y_n$ is the standard basis of $\hh=\CC^n$ and $x_1,\dots,x_n$ is the dual basis in $\hh^*$.  Let $F \supseteq \CC$ be a field containing the complex numbers and fix $\kappa,c_0,d_1,\dots,d_{r-1} \in F$.  In our applications, $F$ will be either $\CC$ or the field of rational functions in the parameters $\kappa,c_0,d_1,\dots,d_{r-1}$ over $\CC$.  We define $d_i$ for all $i \in \ZZ$ by the equations
\begin{equation}
d_0+d_1+\cdots+d_{r-1}=0 \quad \text{and} \quad d_i=d_j \quad \hbox{if $i=j$ mod $r$.}
\end{equation}

The \emph{rational Cherednik algebra} $\HH$ for $G(r,1,n)$ with parameters $\kappa,c_0,d_1,\dots,d_{r-1}$ is the $F$-algebra generated by $F[x_1,\dots,x_n]$, $F[y_1,\dots,y_n]$, and $F G(r,1,n)$ with relations 
\begin{equation*}
t_w x=(wx) t_w \quad \text{and} \quad t_w y=(wy) t_w,
\end{equation*} for $w,v \in W$, $x \in \hh^*$, and $y \in \hh$, 
\begin{equation} \label{cf2}
y_i x_j=x_j y_i+c_0 \sum_{l=0}^{r-1} \zeta^{-l} t_{\zeta_i^l s_{ij} \zeta_i^{-l}},
\end{equation} for $1\leq i \neq j \leq n$, and 
\begin{equation} \label{cf3}
y_i x_i=x_i y_i+\kappa-\sum_{j=0}^{r-1}(d_j-d_{j-1}) e_{ij}-c_0 \sum_{j \neq i} \sum_{l=0}^{r-1} t_{\zeta_i^l s_{ij} \zeta_i^{-l}},
\end{equation} for $1 \leq i \leq n$, where $e_{ij} \in \CC W$ is the idempotent
\begin{equation}
e_{ij}=\frac{1}{r}\sum_{l=0}^{r-1} \zeta^{-lj} t_{\zeta_i^l}.
\end{equation}  The \emph{PBW theorem} (see \cite{Dri}, \cite{EtGi}, \cite{RaSh}, and \cite{Gri2}) for $\HH$ asserts that as $F$-vector spaces,
\begin{equation} \label{PBW theorem}
\HH \simeq F[x_1,\dots,x_n] \otimes_F F W \otimes_F F[y_1,\dots,y_n].
\end{equation}  

The following proposition is obtained from Proposition 2.3 of \cite{Gri2} by specialization to the case $W=G(r,1,n)$.  The parameters in that paper are attached to conjugacy classes of reflections.  If $c_l$ is the parameter attached to the class containing $\zeta_1^l$ then the formula $c_l=\frac{1}{r} \sum_{j=0}^{r-1} \zeta^{-lj} d_j$ (for $l=1,2,\dots,r-1$) relates these parameters to ours.
\begin{proposition} \label{comm formula for y and f}
Let $\mu \in \ZZ_{\geq 0}^n$ and $1 \leq i \leq n$.  Then
\begin{align*}
y_ix^\mu&=x^\mu y_i+\kappa \mu_i x^{\mu-\epsilon_i}-c_0 \sum_{j \neq i} \sum_{l=0}^{r-1}  \frac{x^{\mu}-\zeta_i^l s_{ij} \zeta_i^{-l} x^{\mu}}{x_i-\zeta^l x_j} t_{\zeta_i^l s_{ij} \zeta_i^{-l}} 
 -\sum_{j=0}^{r-1} d_j x^{\mu-\epsilon_i} (e_{i,j}-e_{i,j+\mu_i}),
\end{align*} where $\epsilon_i$ has a $1$ in the $i$th position and $0$'s elsewhere.
\end{proposition}

In \cite{DuOp} section 3, Dunkl and Opdam defined elements
\begin{equation}
z_i=y_i x_i+c_0 \phi_i \quad \text{where} \quad \phi_i=\sum_{1 \leq j <i} \sum_{l=0}^{r-1} t_{\zeta_i^l s_{ij} \zeta_i^{-l}}
\end{equation} and proved that they are pairwise commutative.  By use of \ref{cf3} we have
\begin{equation}
z_i=x_i y_i+\kappa-\sum_{j=0}^{r-1} (d_j-d_{j-1}) e_{ij} -c_0 \sum_{\substack{i < j \leq n \\ 0 \leq l \leq r-1}} t_{\zeta_i^l s_{ij} \zeta_i^{-l}}.
\end{equation}  Let $\ttt$ be the (commutative) subalgebra of $\HH$ generated by $z_1,\dots,z_n$ and $t_{\zeta_1},\dots,t_{\zeta_n}$.  If $\alpha: \ttt \rightarrow F$ is an $F$-algebra homomorphism and $M$ is an $\HH$-module, define the $\alpha$-\emph{weight space} $M_\alpha$ by
\begin{equation}
M_\alpha=\{m \in M \ | \ \hbox{there is $q \in \ZZ_{>0}$ such that }(f-\alpha(f))^q.m=0 \quad \hbox{for all $f \in \ttt$.} \}
\end{equation}   If $v \in M$, we say that $v$ has $\ttt$-weight $(\alpha_1,\dots,\alpha_n,\zeta^{\beta_1},\dots,\zeta^{\beta_n})$ if
\begin{equation} \label{weight conditions}
z_i.v=\alpha_i v \quad \text{and} \quad t_{\zeta_i}.v=\zeta^{\beta_i} v 
\quad \hbox{for $1 \leq i \leq n$.}
\end{equation}  If $v \neq 0$ then the sequence $(\alpha_1,\dots,\alpha_n,\zeta^{\beta_1},\dots,\zeta^{\beta_n})$ is determined by \eqref{weight conditions}.  As observed in \cite{Dez}, the subalgebra of $\HH$ generated by $\ttt$ and $\CC W$ is the generalized graded affine Hecke algebra for $G(r,1,n)$ introduced by Ram and Shepler in \cite{RaSh}.

The \emph{intertwining operators} $\sigma_i$ are the operators
\begin{equation} \label{sigma def}
\sigma_i=t_{s_i}+\frac{c_0}{z_i-z_{i+1}} \pi_i \quad \hbox{for $1 \leq i \leq n-1$.}
\end{equation}  The operator $\sigma_i$ is defined on those $\ttt$-weights spaces $M_\alpha$ on which $z_i-z_{i+1}$ is invertible or $\pi_i$ is zero.  They satisfy the relations
\begin{equation}
z_i \sigma_i=\sigma_i z_{i+1}, \ z_i \sigma_j=\sigma_j z_i \quad \hbox{if $j \neq i-1,i$,}
\end{equation}
\begin{equation}
t_{\zeta_i} \sigma_i=\sigma_i t_{\zeta_{i+1}}, \ t_{\zeta_i} \sigma_j=\sigma_j t_{\zeta_i} \quad \hbox{if $j \neq i-1,i$,}
\end{equation}
\begin{equation}
\sigma_i \sigma_{i+1} \sigma_i=\sigma_{i+1} \sigma_i \sigma_{i+1}, \quad \text{and} \quad \sigma_i^2=\frac{(z_i-z_{i+1}-c_0 \pi_i)(z_i-z_{i+1}+c_0 \pi_i)}{(z_i-z_{i+1})^2},
\end{equation} for $1 \leq i \leq n-2$ and $1 \leq i \leq n-1$, all of which can be checked by straightforward (sometimes lengthy) calculations.

We also define the intertwining operators $\Phi$ and $\Psi$ by
\begin{equation} \label{phipsi defs}
\Phi=x_n t_{s_{n-1} \cdots s_1} \quad \text{and} \quad \Psi=y_1 t_{s_1 \cdots s_{n-1}}.
\end{equation}  The intertwiner $\Phi$ was discovered by Knop and Sahi (\cite{KnSa}).   
Put
\begin{equation} \label{phi def 1}
\phi.(\alpha_1,\dots,\alpha_n,\zeta^{\beta_1},\dots,\zeta^{\beta_n})=(\alpha_2,\dots,\alpha_n,\alpha_1+\kappa-d_{\beta_1-1}+d_{\beta_1-2},\zeta^{\beta_2},\dots,\zeta^{\beta_n},\zeta^{\beta_1-1}). 
\end{equation}  Let $\psi=\phi^{-1}$, so that
\begin{equation}
\psi.(\alpha_1,\dots,\alpha_n,\zeta^{\beta_1},\dots,\zeta^{\beta_n})=((\alpha_n-\kappa+d_{\beta_n}-d_{\beta_n-1},\alpha_1,\dots,\alpha_{n-1},\zeta^{\beta_n+1},\zeta^{ \beta_1},\dots,\zeta^{\beta_{n-1}}).
\end{equation}  Let $S_n$ act on weights by simultaneously permuting the $\alpha_i$'s and $\zeta^{\beta_i}$'s.

The following proposition, proved in \cite{Gri2} Lemmas 5.2 and 5.3, records the properties satisfied by the intertwining operators $\Phi, \Psi$, and $\sigma_i$ for $1 \leq i \leq n-1$.  
\begin{proposition} \label{intertwiner properties}
Let $M$ be an $\HH$-module and suppose $m \in M$ has $\ttt$-weight $\text{wt}(m)=(\alpha_1,\dots,\alpha_n,\zeta^{\beta_1},\dots,\zeta^{\beta_n})$.  Then
\begin{enumerate}
\item[(a)]  If $\zeta^{\beta_i} \neq \zeta^{\beta_{i+1}}$ or $a_{i+1} \neq a_i$ then $\sigma_i.m$ is well-defined with weight
\begin{equation*}
\text{wt}(\sigma_i.m)=s_i.\text{wt}(m).
\end{equation*}  

\item[(b)] The weights of $\Phi.m$ and $\Psi.m$ are
\begin{equation*}
\text{wt}(\Phi.m)=\phi.\text{wt}(m) \quad \text{and} \quad \text{wt}(\Psi.m)=\psi.\text{wt}(m).
\end{equation*}

\item[(c)]  $\Psi \Phi=z_1$, and

\item[(d)] \begin{equation*}
\Phi \Psi=z_n-\kappa+\sum_{j=0}^{r-1} (d_j-d_{j-1}) e_{nj}.
\end{equation*} 
\end{enumerate}
\end{proposition}

\section{Spectrum of $M(\lambda)$.} \label{spectrum section}

Recall from Theorem \ref{CW modules theorem} that the irreducible $\CC W$-modules $S^{\lambda}$ are parametrized by $r$-partitions $\lambda$ of $n$.  Define the Verma module $M(\lambda)$ to be the induced module
\begin{equation}
M(\lambda)=\text{Ind}_{F W \otimes F[y_1,\dots,y_n]}^\HH S^{\lambda},
\end{equation} where by abuse of notation $S^{\lambda}$ is the $F W$-module obtained by extension of scalars to $F$ and we define the $F[y_1,\dots,y_n]$ action on $S^{\lambda}$ by
\begin{equation}
y_i.v=0 \quad \hbox{for $1 \leq i \leq n$ and $v \in S^{\lambda}$.}
\end{equation}  By the PBW theorem \eqref{PBW theorem} for $\HH$ we have an isomorphism of $F$-vector spaces
\begin{equation}
M(\lambda) \simeq F[x_1,\dots,x_n] \otimes_F S^{\lambda}.
\end{equation}  As usual let $\text{SYT}(\lambda)$ be the set of standard Young tableaux on $\lambda$.  We will show that for generic choices of the parameters $\kappa$ and $c_{l}$, the standard module $M(\lambda)$ has a basis of $\ttt$-eigenvectors indexed by the set $\ZZ_{\geq 0}^n \times \text{SYT}(\lambda)$ and we will calculate the eigenvalues explicitly.

We define a partial order on $\ZZ_{\geq 0}^n$ by 
\begin{equation}
\mu < \nu \quad \iff \quad \mu_+ <_d \nu_+ \quad \text{or} \quad \mu_+=\nu_+ \quad \text{and} \quad w_\mu < w_\nu,
\end{equation}  where we use the Bruhat order on $S_n$, and $<_d$ is dominance order on partitions:
\begin{equation}
\mu>_d \nu \quad \hbox{if $\mu-\nu=\sum_{i=1}^{n-1} k_i (\epsilon_i-\epsilon_{i+1})$ with $k_i \geq 0$ for $1 \leq i \leq n-1$.}
\end{equation}  If $\mu \in \ZZ_{\geq 0}^n$, $1 \leq i<j \leq n$, and $\mu_i>\mu_j$, then with respect to this partial order one has
\begin{equation} \label{order fact}
\mu > s_{ij}.\mu+k(\epsilon_i-\epsilon_j) \quad \hbox{for $0 \leq k <  \mu_i-\mu_j$.}
\end{equation}

Recall from \eqref{JM basis} that for any $\mu \in \ZZ_{\geq 0}^n$ the elements $v_T^\mu$ are a basis of $S^{\lambda}$.  The next theorem is the analogue of \cite{Opd} 2.6 in our setting, showing that the $z_i$'s are upper triangular as operators on $M(\lambda)$ with respect to the basis $x^\mu v_T^\mu$ of $M(\lambda)$ and the order on $\ZZ_{\geq 0}^n \times \text{SYT}(\lambda)$ defined by
\begin{equation}
(\mu,T)<(\nu,S) \quad \iff \quad \mu < \nu.
\end{equation}
\begin{theorem} \label{Upper triangular}
Let $\lambda$ be an $r$-partition of $n$, $\mu \in \ZZ_{\geq 0}^n$, and let $T$ be a standard tableau on $\lambda$.  Recall the definitions of $\beta$ and ct given in \eqref{content def} and \eqref{beta def}.
\begin{enumerate}
\item[(a)]  The action of $t_{\zeta_i}$ and $z_i$ on $M(\lambda)$ are given by
\begin{align*}
t_{\zeta_i}.x^\mu v_T^\mu=\zeta^{\beta(T(w_\mu(i)))-\mu_i} x^\mu v_T^\mu
\end{align*}
and
\begin{align*} \label{z spectrum}
z_i.x^\mu v_T^\mu&=\left( \kappa(\mu_i+1)-(d_{\beta(T(w_\mu(i)))}-d_{\beta(T(w_\mu(i)))-\mu_i-1})-c_0 r \text{ct}(T(w_\mu(i))) \right) x^\mu v_T^\mu \\
&+\sum_{(\nu,S) < (\mu,T)} c_{\nu,S} x^\nu v_{S}^\nu.
\end{align*}
\item[(b)]  Assuming that the parameters are generic, so $F=\CC(\kappa,c_0,d_1,d_2,\dots,d_{r-1})$, for each $\mu \in \ZZ_{\geq 0}^n$ and $T \in \text{SYT}(\lambda)$ there exists a unique $\ttt$ eigenvector $f_{\mu,T} \in M(\lambda)$ such that
\begin{equation*}
f_{\mu,T}=x^\mu v_T^\mu+\text{lower terms}.
\end{equation*}  
\end{enumerate}
The $\ttt$-eigenvalue of $f_{\mu,T}$ is determined by the formulas in part (a).
\end{theorem}  
\begin{proof}
The statement about the action of $t_{\zeta_i}$ follows from the commutation relation in the definition of the rational Cherednik algebra and the definition of the representation $M(\lambda)$.  Using the commutation formula in Proposition \ref{comm formula for y and f} for $f \in \CC[x_1,\dots,x_n]$ and $y \in \hh$ and the geometric series formula to evaluate the divided differences, we obtain the following formula for the action of $y_i x_i$ on $x^\mu v_T^\mu$, in which $b_i=\beta(T(w_\mu(i)))$ and $a_i= r \text{ct}(T(w_\mu(i)))$:
\begin{align*}
y_i.x^{\mu+\epsilon_i}v_T^\mu&=\left(\kappa(\mu_i+1) x^{\mu}-c_0 \sum_{j \neq i} \sum_{l=0}^{r-1}  \frac{x^{\mu+\epsilon_i}-\zeta_i^l s_{ij} \zeta_i^{-l} x^{\mu+\epsilon_i}}{x_i-\zeta^l x_j} t_{\zeta_i^l s_{ij} \zeta_i^{-l}} 
 -\sum_{j=0}^{r-1} d_j x^{\mu} (e_{i,j}-e_{i,j+\mu_i+1}) \right) v_T^\mu \\
 &=\Big( \kappa(\mu_i+1)x^\mu-(d_{b_i}-d_{b_i-\mu_i-1}) x^\mu +c_0 \sum_{\substack{j \neq i \\ \mu_j > \mu_i \\ 0 \leq l \leq r-1}} \sum_{k=1}^{\mu_j-\mu_i-1} \zeta^{-lk} x^{\mu+k(\epsilon_i-\epsilon_j)} t_{\zeta_i^l s_{ij} \zeta_i^{-l}} \\
 &- c_0 \sum_{\substack{j \neq i \\ \mu_i \geq \mu_j \\ 0 \leq l \leq r-1 }} (x^\mu+ \zeta^l x^{\mu+(\epsilon_j-\epsilon_i)}+ \cdots +\zeta^{l(\mu_i-\mu_j)} x^{ s_{ij} \mu} ) t_{\zeta_i^l s_{ij} \zeta_i^{-l}}  \Big) v_T^\mu.
\end{align*}

Using this equation and \eqref{order fact} to identify lower terms,
\begin{align*}
z_i .x^\mu v_T^\mu&=\left(y_i x_i +c_0 \sum_{\substack{1 \leq j <i \\ 0 \leq l \leq r-1}} t_{\zeta_i^l s_{ij} \zeta_i^{-l}} \right) x^\mu v_T^\mu
=y_i. x^{\mu+\epsilon_i}v_T^\mu+c_0 \sum_{\substack{1 \leq j < i \\ 0 \leq l \leq r-1}} t_{\zeta_i^{l} s_{ij} \zeta_i^{-l}}. x^\mu v_T^\mu \\
&=\Big( (\kappa(\mu_i+1)-(d_{b_i}-d_{b_i-\mu_i-1})) x^{\mu}
 - c_0 \sum_{\substack{1 \leq j < i \\ \mu_j \leq \mu_i \\ 0 \leq l \leq r-1}} \zeta^{l(\mu_i-\mu_j)} x^{ s_{ij} \mu} t_{\zeta_i^l s_{ij} \zeta_i^{-l}} \\
 & - c_0 \sum_{\substack{ 1 \leq j < i \\ \mu_j < \mu_i \\ 0 \leq l \leq r-1}} x^\mu t_{\zeta_i^l s_{ij} \zeta_i^{-l}}  
 - c_0 \sum_{\substack{ i < j \leq n \\ \mu_j \leq \mu_i \\ 0 \leq l \leq r-1}}x^\mu t_{\zeta_i^l s_{ij} \zeta_i^{-l}} 
 +c_0  \sum_{\substack{1 \leq j < i \\ 0 \leq l \leq r-1}} \zeta^{l(\mu_i-\mu_j)} x^{ s_{ij} \mu} t_{\zeta_i^l s_{ij} \zeta_i^{-l}} \Big ) v_T^\mu \\
  &+ \text{lower terms} \\
 &= \Big( (\kappa(\mu_i+1)-(d_{b_i}-d_{b_i-\mu_i-1})) x^{\mu}
 - c_0 x^\mu \phi_i^\mu
 +c_0  \sum_{\substack{1 \leq j < i \\ \mu_j > \mu_i \\ 0 \leq l \leq r-1}} \zeta^{l(\mu_i-\mu_j)} x^{ s_{ij} \mu} t_{\zeta_i^l s_{ij} \zeta_i^{-l}} \Big ) v_T^\mu \\
  &+ \text{lower terms} \\
  &= \left( \kappa(\mu_i+1)-(d_{b_i}-d_{b_i-\mu_i-1})
 - c_0  a_i
 \right) x^\mu v_T^\mu + \text{lower terms}.
\end{align*}  This proves part (a).

For part (b), we will prove using part (a) that the (a priori generalized) $\ttt$-eigenspaces on $M(\lambda)$ are one-dimensional.  Suppose there are $\mu,\nu \in \ZZ_{\geq 0}^n$ and $T,S \in \text{SYT}(\lambda)$ with
\begin{equation} \label{b equation1}
\zeta^{\beta(T(w_\mu(i)))-\mu_i}=\zeta^{\beta(T(v_\nu(i)))-\nu_i},
\end{equation} and
\begin{align*}
\kappa(\mu_i+1)&-(d_{\beta(T(w_\mu(i)))}-d_{\beta(T(w_\mu(i)))-\mu_i-1})-c_0 r \text{ct}(T(w_\mu(i))) \\
&=\kappa(\nu_i+1)- (d_{\beta(S(v_\nu(i)))}-d_{\beta(S(v_\nu(i)))-\nu_i-1})-c_0 r \text{ct}(S(v_\nu(i)))
\end{align*} for $1 \leq i \leq n$.  By comparing coefficients of $\kappa$ in the equation above we find $\mu_i=\nu_i$ for $1 \leq i \leq n$.  Next, comparing coefficients of $c_0$ implies that $\text{ct}(T(i))=\text{ct}(S(i))$ for $1 \leq i \leq n$.  Finally \eqref{b equation1} implies that the sequences $\zeta^{\beta(T(1))},\dots,\zeta^{\beta(T(n))}$ and $\zeta^{\beta(S((1))},\dots,\zeta^{\beta(S(n))}$ are equal, and part (b) follows.  
\end{proof}

As a simple application of Theorem~\ref{Upper triangular}, we show how dominance order and certain $m$-cores arise naturally from the representation theory of the rational Cherednik algebra of type $G(1,1,n)$.  It seems likely that the orders described in Gordon's paper \cite{Gor2} arise in this way for $r>1$; if this is the case then a positive answer to Question 10.1 of Gordon's paper should follow.  For $c_0 \notin \ZZ+\frac{1}{2}$, the next corollary follows from Rouquier's work \cite{Rou} and the corresponding result for the $q$-Schur algebra.  

\begin{corollary}
Suppose $r=1$ and $c_0=k/m$ for relatively prime positive integers $k$ and $m$.  If $L(\mu)$ occurs as a composition factor of $M(\lambda)$ then $\mu \leq_d \lambda$ and the $m$-core of $\mu$ is the same as the $m$-core of $\lambda$.  
\end{corollary}
\begin{proof}
If $L(\mu)$ occurs as a composition factor of $M(\lambda)$, then the $\ttt$-weights of the subspace $S^\mu \subseteq M(\mu)$ must occur among the $\ttt$-weights of $M(\lambda)$.  Therefore by part (a) of Theorem~\ref{Upper triangular} there exist integers $\mu_i \in \ZZ_{\geq 0}$ and orderings $b_1,b_2,\dots,b_n$ and $b'_1,b'_2,\dots,b'_n$ of the boxes of $\lambda$ and $\mu$ so that
\begin{equation}
\mu_i=\left(\text{ct}(b_i)-\text{ct}(b'_i) \right) k/m.
\end{equation}  The first consequence is that $\text{ct}(b_i)=\text{ct}(b_i')$ mod $m$ for $1 \leq i \leq n$, which implies that the $m$-cores of $\lambda$ and $\mu$ are equal by Theorem 2.7.41 of \cite{JaKe}.  The second consequence is that $\text{ct}(b_i) \geq \text{ct}(b_i')$ for $1 \leq i \leq n$, whence $\lambda \geq_d \mu$.  
\end{proof}

The simultaneous eigenfunctions $f_{\mu,T}$ are a generalization of the non-symmetric Jack polynomials: in the special case when $W=G(1,1,n)$ and $\lambda=(n)$ one obtains the usual non-symmetric Jack polynomials, and when $W=G(r,1,n)$ and $\lambda=(n),\emptyset,\dots,\emptyset$ one obtains the polynomials discovered in section 3 of \cite{DuOp} (which are a slight modification of non-symmetric Jack polynomials).  In Section~\ref{Norm section}, we introduce a certain inner product on $M(\lambda)$ such that $z_i$ is self-adjoint and $t_{\zeta_i}$ is unitary.  With respect to this inner product, the functions $f_{\mu,T}$ are pairwise orthogonal since they have distinct $\ttt$-eigenvalues.  For non-trivial $r$-partitions $\lambda$ one therefore obtains new orthogonal functions, and it should be interesting to investigate their properties, especially for the case of the symmetric group $G(1,1,n)$.

For $(\mu_1,\dots,\mu_n) \in \ZZ_{\geq 0}^n$, define
\begin{equation} \label{phi def 2}
\phi.(\mu_1,\mu_2,\dots,\mu_n)=(\mu_2,\mu_3,\dots,\mu_1+1) \quad \text{and} \quad \psi.(\mu_1,\dots,\mu_n)=(\mu_n-1,\mu_1,\dots,\mu_{n-1}).
\end{equation}  These operators on $\ZZ^n$ will turn out to correspond to the intertwiners $\Phi$ and $\Psi$.  The following fundamental lemma describes how the intertwining operators act on the basis $f_{\mu,T}$ of $M(\lambda)$.  If the parameters are specialized in such a way that the spectrum of $M(\lambda)$ remains simple, then it allows one to give an explicit description of the submodule structure of $M(\lambda)$: see Theorem \ref{submodules theorem}.  We will also use it to prove the norm formula for $f_{\mu,T}$ in Theorem \ref{norm formula}.

\begin{lemma} \label{action lemma}
Let $\mu \in \ZZ_{\geq 0}^n$ and let $T$ be a standard Young tableau on $\lambda$.
\begin{enumerate}
\item[(a)] Suppose $\mu_i \neq \mu_{i+1}$.  If $\mu_i<\mu_{i+1}$ or $\mu_i-\mu_{i+1} \neq \beta(T(w_\mu(i)))-\beta(T(w_\mu(i+1)))$ mod $r$ then
\begin{equation*}
\sigma_i.f_{\mu,T}=f_{s_i.\mu,T}.
\end{equation*}
\item[(b)]  If $\mu_i>\mu_{i+1}$ and $\mu_i-\mu_{i+1} =\beta(T(w_\mu(i)))-\beta(T(w_\mu(i+1)))$ mod $r$ then
\begin{equation*}
\sigma_i.f_{\mu,T}=\frac{(\delta-rc_0)(\delta+rc_0)}{\delta^2} f_{s_i \mu,T},
\end{equation*} where
\begin{equation*}
\delta=\kappa(\mu_{i}-\mu_{i+1})-(d_{\beta(T(w_\mu(i)))}-d_{\beta(T(w_\mu(i+1)))})-c_0 r (\text{ct}(T(w_\mu(i)))-\text{ct}(T(w_\mu(i+1)))).
\end{equation*}
\item[(c)] Put $j=w_\mu(i)$.  If $\mu_i=\mu_{i+1}$ then
\begin{equation*}
\sigma_i.f_{\mu,T}=\begin{cases}  0 \quad &\hbox{if $s_{j-1}.T$ is not a standard tableau,} \\
f_{\mu,s_{j-1}.T} \quad &\hbox{if $\zeta^{\beta(T(j))}\neq \zeta^{\beta(T(j-1))}$,} \\ 
\left(1-\left( \frac{1}{\text{ct}(T(j-1))-\text{ct}(T(j))} \right)^2 \right)^{1/2} f_{\mu,s_{j-1}.T} \quad &\text{else.} \end{cases}
\end{equation*} 
\item[(d)]  For all $\mu \in \ZZ_{\geq 0}^n$,
\begin{equation*}
\Phi.f_{\mu,T}=f_{\phi.\mu,T}.
\end{equation*}
\item[(e)]  For all $\mu \in \ZZ_{\geq 0}^n$,
\begin{equation*}
\Psi.f_{\mu,T}=\begin{cases}  \left(\kappa \mu_n-(d_{\beta(T(w_\mu(n)))}-d_{\beta(T(w_\mu(n)))-\mu_n})-r \text{ct}(T(w_\mu(n))) c_0 \right) f_{\psi.\mu,T}\quad &\hbox{if $\mu_n>0$,} \\ 0 \quad &\hbox{if $\mu_n=0$.} \end{cases}
\end{equation*}
\end{enumerate}
\end{lemma}
\begin{proof}
The method of proof for parts (a), (b), and (c) of the lemma is the same: one checks that both sides of the equation are $\ttt$-eigenvectors with the same leading term and applies part (b) of Theorem \ref{Upper triangular}.  We use the fact that for generic parameters the intertwiners $\sigma_i$ are well-defined on all $f_{\mu,T}$'s, and that by Proposition \ref{intertwiner properties} an intertwiner applied to a $\ttt$-eigenvector is a $\ttt$-eigenvector if it is non-zero.

For (a) we observe that if $\mu_i<\mu_{i+1}$ then
\begin{align*}
\sigma_i.f_{\mu,T}&=(t_{s_i}+\frac{c_0 \pi_i}{z_i-z_{i+1}}).(x^\mu v_T^\mu+\text{lower terms})
=x^{s_i.\mu} t_{s_i}. v_T^\mu+\text{lower terms} \\
&=x^{s_i.\mu} t_{s_i w_\mu w_0}.v_T+\text{lower terms} 
=x^{s_i.\mu} v_T^{s_i.\mu}+\text{lower terms}.
\end{align*} This implies that $\sigma_i.f_{\mu,T}=f_{s_i.\mu,T}$.  On the other hand, if $\mu_i>\mu_{i+1}$ and $\mu_i-\beta(T(w_\mu(i))) \neq \mu_{i+1}-\beta(T(w_\mu(i+1)))$ mod $r$ then using the previous calculation gives
\begin{equation}
\sigma_i.f_{\mu,T}=\sigma_i^2 f_{s_i \mu, T}=t_{s_i}^2.f_{s_i\mu,T}=f_{s_i \mu, T}.
\end{equation}  This proves (a).

For (b) we assume that $\mu_i>\mu_{i+1}$ and  $\mu_i-\beta(T(w_\mu(i))) = \mu_{i+1}-\beta(T(w_\mu(i)))$ mod $r$ and compute using (a), Proposition \ref{Upper triangular}, and Proposition \ref{intertwiner properties},
\begin{align*}
\sigma_i.f_{\mu,T}=\sigma_i^2.f_{s_i \mu,T}
=\frac{(z_i-z_{i+1}-c_0 \pi_i)(z_i-z_{i+1}+c_0 \pi_i)}{(z_i-z_{i+1})^2}.f_{s_i.\mu,T} 
=\frac{(\delta-rc_0)(\delta+rc_0)}{\delta^2} f_{s_i \mu,T},
\end{align*} where
\begin{equation}
\delta=\kappa(\mu_{i}-\mu_{i+1})-(d_{\beta(T(w_\mu(i)))}-d_{\beta(T(w_\mu(i+1)))})-c_0 r (\text{ct}(T(w_\mu(i)))-\text{ct}(T(w_\mu(i+1))))
\end{equation} is the scalar by which $z_i-z_{i+1}$ acts on $f_{s_i.\mu,T}$.  This proves (b).

Assuming that $\mu_i=\mu_{i+1}$ and writing $w=w_\mu^{-1}$,  Theorem \ref{Upper triangular} and the formula in part (c) of Theorem \ref{CW modules theorem} for the action of the intertwiners $\tau_i$ on $S^\lambda$ give
\begin{align*}
\sigma_i.f_{\mu,T}&=\left(t_{s_i}+\frac{\pi_i c_0}{z_i-z_{i+1}} \right).\left(x^\mu v_T^\mu+\text{lower terms} \right) \\
&=x^\mu \left( t_{s_i}+\frac{c_0}{c_0 r \text{ct}(T(w_\mu(i+1)))-c_0 r \text{ct}(T(w_\mu(i)))} \pi_i \right).t_w.v_T+\text{lower terms} \\
&=x^\mu t_{w}.\left( t_{s_{w^{-1}(i)-1}}+\frac{1}{r \text{ct}(T(w^{-1}(i+1)))- r\text{ct}(T(w^{-1}(i)))} \pi_{w^{-1}(i)-1} \right).v_T+\text{lower terms} \\
&=x^\mu t_w \tau_{j-1} v_T+\text{lower terms}
\end{align*} where $j=w^{-1}(i)=w_\mu(i)$.  Combined with Theorem \ref{CW modules theorem} this proves (c).

For (d), since $\phi$ does behave as well as $s_i$ with respect to our order on $\ZZ_{\geq 0}^n$ one first checks that
\begin{equation} \label{phi wt equation}
\text{wt}(f_{\phi.\mu,T})=\phi.\text{wt}(f_{\mu,T}),
\end{equation} where $\phi.\mu$ is defined in \eqref{phi def 2} and $\phi.\text{wt}(f_{\mu,T})$ is defined in \eqref{phi def 1}.  This is a straightforward (but somewhat tedious) calculation that we omit.

Then the equation
\begin{equation}
\Phi.f_{\mu,T}=x_n t_{s_{n-1} s_{n-2} \cdots s_1}.(x^\mu v_T^\mu+\text{lower terms})=x^{\phi.\mu}  t_{s_{n-1} s_{n-2} \cdots s_1}.v_T^\mu+\sum_{\nu \neq \phi.\mu} c_{\nu,T'} x^\nu v_{T'}^\nu
\end{equation} implies the equality in (d).  

Turning to (e), we first observe that if $\mu_n=0$ then by part (b) of Proposition \ref{intertwiner properties}
\begin{align*}
z_1.\Psi.f_{\mu,T}&=\left(\kappa-(d_{\beta(T(w_\mu(n)))}-d_{\beta(T(w_\mu(n)))-1})-c_0 r \text{ct}(T(w_\mu(n)))-\kappa+d_{\beta(T(w_\mu(n)))}-d_{\beta(T(w_\mu(n)))-1} \right) \Psi.f_{\mu,T} \\
&=-c_0 r \text{ct}(T(w_\mu(n))) \Psi.f_{\mu,T}
\end{align*}  By part (a) of Theorem \ref{Upper triangular} the weights of $z_1$ on $M(\lambda)$ all have positive coefficient on $\kappa$, and it follows that $-c_0 r \text{ct}(T(w_\mu(n)))$ is not a weight of $z_1$ on $M(\lambda)$.  Hence $\Psi.f_{\mu,T}=0$.  If $\mu_n>0$ then we compute using part (d) and Proposition \ref{intertwiner properties}
\begin{align*}
\Psi.f_{\mu,T}&=\Psi.\Phi.f_{\psi.\mu,T}=z_1.f_{\psi.\mu} \\
&=\left(\kappa \mu_n-(d_{\beta(T(v_{\psi.\mu}(1)))}-d_{\beta(T(v_{\psi.\mu}(1)))-(\psi.\mu)_1-1})-c_0 r \text{ct}(T(v_{\psi.\mu}(1))) \right) f_{\psi.\mu,T} \\
&=\left(\kappa \mu_n-(d_{\beta(T(w_\mu(n)))}-d_{\beta(T(w_\mu(n)))-\mu_n})-c_0 r \text{ct}(T(w_\mu(n))) \right) f_{\psi.\mu,T} \ 
\end{align*} 
\end{proof}

As a first application of the Lemma, we prove a formula that should be useful in analyzing the restrictions $f_{\mu,T}$ to $x_i=x_{i+1}=\cdots=x_n=0$.  It will not be used in the rest of this paper.

\begin{corollary}
If $1 \leq i \leq n$ and $\mu_j=0$ for $j \geq i$ then
\begin{equation}
y_i.f_{\mu,T}=0. 
\end{equation}
\end{corollary}
\begin{proof}
By the definition \eqref{sigma def} of the intertwiners $\sigma_j$ and part (c) of Lemma \ref{action lemma} we have
\begin{equation*}
t_{s_{n-1} s_{n-2} \cdots s_i}.f_{\mu,T} \in F \text{-span} \{f_{\mu,S} \ | \ \hbox{$S$ a standard tableau on $\lambda$} \},
\end{equation*} and consequently part (e) of Lemma \ref{action lemma} implies
\begin{equation*}
t_{s_1 s_2 \cdots s_{i-1}} y_i.f_{\mu,T}=t_{s_1 s_2 \cdots s_{i-1}} y_i t_{s_i \cdots s_{n-1}} t_{s_{n-1} s_{n-2} \cdots s_i} f_{\mu,T}=\Psi.t_{s_{n-1} s_{n-2} \cdots s_i} f_{\mu,T}=0.
\end{equation*}  Therefore
\begin{equation}
y_i.f_{\mu,T}=0 \quad \hbox{if $\mu_j=0$ for $i \leq j \leq n$.} \  
\end{equation}
\end{proof} 

\section{Norm formula} \label{Norm section}

In this section we assume that the base field is $F=\CC(\kappa,c_0,d_1,\dots,d_{r-1})$.  Recall that we write $S^{\lambda}$ for the $F W$-module obtained by extension of scalars from $\CC$ to $F$.  We extend complex conjugation to an automorphism of $F$ by fixing $\kappa$, $c_0$, and the $d_j$'s.

Let $\la \cdot, \cdot \ra_0$ be the positive definite $W$-invariant Hermitian form on $S^{\lambda}$ such that $\la v_T,v_T \ra_0=1$ for all standard tableaux $T$ on $\lambda$.  For $x=\sum_{i=1}^n k_i x_i \in \hh^*$ and $y=\sum_{i=1}^n l_i y_i \in \hh$ define 
\begin{equation}
x^*=\sum_{i=1}^n \overline{k_i} y_i \quad \text{and} \quad y^*=\sum_{i=1}^n \overline{l_i} x_i.
\end{equation}  The map $*$ given by $x \mapsto x^*$, $y \mapsto y^*$, and $t_w \mapsto t_{w^{-1}}$ for $x \in \hh^*$, $y \in \hh$ and $w \in W$ extends to a skew-linear---with respect to the automorphism of $F$ defined above---anti-automorphism of $\HH$; this follows directly from the defining relations for $\HH$.  The \emph{contravariant form} on $M(\lambda)$ is the unique Hermitian form $\la \cdot, \cdot \ra$ on $M(\lambda)$ extending the form $\la \cdot, \cdot \ra_0$ on $S^{\lambda}=M(\lambda)^0$ and satisfying
\begin{equation}
\la t_w.f, t_w.g \ra=\la f, g \ra \quad \la x.f,g \ra=\la f, x^*.g \ra \quad \text{and} \quad \la y.f, g \ra=\la f, y^*.g \ra 
\end{equation} for $f,g \in M(\lambda)$, $w \in W$, $x \in \hh^*$, and $y \in \hh$.  The next theorem gives a product formula for the norms of the generalized non-symmetric Jack polynomials with respect to this form.  Note that it is not, generally speaking, a cancelation-free formula; however, if the factors in the denominator do not vanish then the zeros of the numerator control the radical of $M(\lambda)$.
\begin{theorem} \label{norm formula}
For $(\mu,T) \in \ZZ_{\geq 0} \times \text{SYT}(\lambda)$ write
\begin{equation*}
a_i=\text{ct}(T(w_\mu(i))) \quad \text{and} \quad b_i=\beta(T(w_\mu(i))).
\end{equation*}  Then the norm of $f_{\mu,T}$ is given by
\begin{align*}
\la f_{\mu,T}, f_{\mu,T} \ra&=\prod_{i=1}^n \prod_{k=1}^{\mu_i} \left(\kappa k-(d_{b_i}-d_{b_i-k})-c_0 r a_i \right) \\
&\times \prod_{\substack{1 \leq i < j \leq n \\ \mu_i > \mu_j}} \prod_{\substack{1 \leq k \leq \mu_i-\mu_j \\ k=b_i-b_j \ \text{mod} \ r}} \frac{\left(\kappa k-(d_{b_{i}}-d_{b_{j}})-c_0r (a_{i}-a_{j})\right)^2-(c_0 r)^2}{\left(\kappa k-(d_{b_{i}}-d_{b_{j}})-c_0r (a_{i}-a_{j}) \right)^2  } \\
&\times  \prod_{\substack{1 \leq i < j \leq n \\ \mu_i < \mu_j-1}} \prod_{\substack{1 \leq k \leq \mu_j-\mu_i-1 \\ k=b_j-b_i \ \text{mod} \ r}} \frac{\left(\kappa k-(d_{b_{j}}-d_{b_{i}})-c_0r (a_{j}-a_{i})\right)^2-(c_0 r)^2}{\left(\kappa k-(d_{b_{j}}-d_{b_{i}})-c_0r(a_{j}-a_{i}) \right)^2  }.
\end{align*} 
\end{theorem}
\begin{proof}
First observe that the operators $z_i$ are self-adjoint with respect to the contravariant form:
\begin{equation}
z_i^*=(y_i x_i)^*+\phi_i^*=x_i^* y_i^*+\phi_i=y_i x_i+\phi_i=z_i,
\end{equation} and hence 
\begin{equation}
\sigma_i^*=\left( t_{s_i}+\frac{c_0 \pi_i}{z_i-z_{i+1}} \right) ^*=t_{s_i}+\frac{c_0 \pi_i}{z_i-z_{i+1}}=\sigma_i
\end{equation} for $1 \leq i \leq n-1$ and
\begin{equation}
\Phi^*=(x_n t_{s_{n-1} \cdots s_1} )^*=t_{s_1 \cdots s_{n-1} } y_n=\Psi.
\end{equation}

Combining these formulas with Lemma \ref{action lemma} shows that for $\mu_n>0$ we have
\begin{align*}
\la f_{\mu,T},f_{\mu,T} \ra&=\la \Phi.f_{\psi.\mu,T},\Phi.f_{\psi.\mu,T} \ra=\la f_{\psi.\mu,T},\Psi \Phi.f_{\psi.\mu,T} \ra=\la f_{\psi.\mu,T}, z_1 f_{\psi.\mu,T} \ra \\
&=\left(\kappa \mu_n-(d_{b_n}-d_{b_n-\mu_n})-c_0 r a_n \right) \la f_{\psi.\mu,T},f_{\psi.\mu,T} \ra,
\end{align*} if $\mu_i>\mu_{i+1}$ and $\zeta^{b_{i+1}-b_i} \neq \zeta^{\mu_{i+1}-\mu_i}$ then 
\begin{equation*}
\la f_{\mu,T}, f_{\mu,T} \ra=\la \sigma_i.f_{s_i.\mu,T},\sigma_i.f_{s_i.\mu,T} \ra=\la f_{s_i.\mu,T},\sigma_i^2.f_{s_i.\mu,T} \ra=\la f_{s_i.\mu,T},f_{s_i.\mu,T} \ra
\end{equation*} and if $\mu_i>\mu_{i+1}$ and $\zeta^{b_{i+1}-b_i} = \zeta^{\mu_{i+1}-\mu_i}$ then
\begin{align*}
\la f_{\mu,T}, f_{\mu,T} \ra&=\la \sigma_i.f_{s_i.\mu,T},\sigma_i.f_{s_i.\mu,T} \ra=\la f_{s_i.\mu,T},\sigma_i^2.f_{s_i.\mu,T} \ra \\
&=\frac{\left(\kappa(\mu_{i}-\mu_{i+1})-(d_{b_{i}}-d_{b_{i+1}})-c_0 r (a_{i}-a_{i+1})\right)^2-(c_0 r)^2}{\left(\kappa(\mu_{i}-\mu_{i+1})-(d_{b_{i}}-d_{b_{i+1}})-c_0r(a_{i}-a_{i+1}) \right)^2  } \la f_{s_i.\mu,T} f_{s_i.\mu,T} \ra.
\end{align*}  The theorem is proved using these formulas by a straightforward induction on $\mu_1+\mu_2+\cdots+\mu_n+l$, where $l$ is the length of the shortest permutation $v$ such that $v.\mu$ is in non-decreasing order.
\end{proof}

\begin{remark}
The radical of the form $\la \cdot,\cdot \ra$ is the unique maximal submodule of $M(\lambda)$, and therefore $M(\lambda)$ is simple exactly if the form is non-degenerate.  In the next section we give much more precise information on the submodule structure of $M(\lambda)$.  
\end{remark}

\section{A hyperplane arrangement and submodules of $M(\lambda)$} \label{submodules section}

In this section we assume $\kappa=1$; after proving Lemma 7.1 we will also assume $c_0 \neq 0$ (if $c_0=0$ the algebra $\HH$ becomes much simpler; see e.g. \cite{EtMo}).  Our goal is to give an explicit description of the submodule structure of $M(\lambda)$ in those cases for which $\ttt$ has simple spectrum on $M(\lambda)$.  First we describe the hyperplanes the parameters $(c_0,d_1,\dots,d_{r-1})$ must avoid in order for the $\ttt$-eigenspaces of $M(\lambda)$ to be one-dimensional. We define the set $\mathcal{E}_\lambda$ of \emph{exceptional hyperplanes} for the $r$-partition $\lambda$ as follows:

Let $\CC^{r}$ be the parameter space for $\HH$ with respect to the parameters $c_0$ and $d_1,d_2,\dots,d_{r-1}$.  Recall that the subscripts $d_i$ are to be read mod $r$ and $d_0=-d_1-\cdots-d_{r-1}$; this defines $d_i$ for all $i \in \ZZ$.  For integers $k,l$ and $m$ define a hyperplane $H_{k,l,m}$ by
\begin{equation}
H_{k,l,m}=\{(c_0,d_1,\dots,d_{r-1}) \ | \ k=d_l-d_{l-k}+m r c_0 \}.
\end{equation}  Then $\mathcal{E}_\lambda$ contains the hyperplanes $H_{k,l,m}$ for all $k \in \ZZ_{>0}$ with $k \neq 0$ mod $r$, $\lambda^l \neq \emptyset$, $\lambda^{l-k} \neq \emptyset$, and 
\begin{equation}
\text{ct}^-(\lambda^l)-\text{ct}^+(\lambda^{l-k}) \leq m \leq \text{ct}^+(\lambda^l)-\text{ct}^-(\lambda^{l-k}),
\end{equation} where for a partition $\nu$, $\text{ct}^+(\nu)$ is the maximum content of a box of $\nu$ and $\text{ct}^-(\nu)$ is the minimum content of a box of $\nu$.  

For each $0 \leq i \leq r-1$ such that $\lambda^i$ is non-empty and all $k \in \ZZ_{>0}$, $\mathcal{E}_\lambda$ contains the hyperplane 
\begin{equation}
k=c_0 m \quad \begin{cases} \hbox{for $0 < m \leq \text{ct}^+(\lambda^i)$ if $\lambda^i$ is a single row,} \\ \hbox{for $\text{ct}^-(\lambda^i) \leq m < 0$ if $\lambda^i$ is a single column, and} \\ \hbox{for $\text{ct}^-(\lambda^i)-\text{ct}^+(\lambda^i) \leq m \leq \text{ct}^+(\lambda^i)-\text{ct}^-(\lambda^i)$ with $m \neq 0$ otherwise.}\end{cases}
\end{equation}

 $\mathcal{E}_\lambda$ is a locally finite set of hyperplanes since every hyperplane it contains is the translate by a distance bounded away from $0$ of one of the form
\begin{equation}
k=d_l-d_{l-k}+m r c_0 
\end{equation} where $0 \leq k,l \leq r-1$ and $m$ runs over a finite set of integers (depending on $\lambda$).  

For an element $(\mu,T) \in \ZZ_{\geq 0}^n \times \text{SYT}(\lambda)$, define
\begin{equation*}
\text{wt}(\mu,T)_i=(\mu_i+1-(d_{\beta(T(w_\mu(i)))}-d_{\beta(T(w_\mu(i)))-\mu_i-1})-c_0 r \text{ct}(T(w_\mu(i))),\zeta^{\beta(T(w_\mu(i)))-\mu_i}) \quad \hbox{for $1 \leq i \leq n$.}
\end{equation*} 

\begin{lemma} \label{simple spectrum}  Suppose $c_0 \neq 0$.  The following are equivalent:
\begin{enumerate}
\item[(a)]  The $\ttt$-eigenspaces on $M(\lambda)$ are all one dimensional;
\item[(b)] for all $(\mu,T) \in \Gamma$, $\text{wt}(\mu,T)_i \neq \text{wt}(\mu,T)_{i+1}$ for $1 \leq i \leq n-1$;
\item[(c)] the point $(c_0,d_1,\dots,d_{r-1})$ does not lie in any hyperplane $H \in \mathcal{E}_\lambda$;
\item[(d)] the functions $f_{\mu,T}$ are all well-defined;
\item[(e)] the intertwiners $\sigma_i$ are well-defined on all of $M(\lambda)$.
\end{enumerate}
\end{lemma}
\begin{proof}
The equivalence of (b), (c), (d), and (e) follow from the definitions and the observation that if $f_{\mu,T}$ is well-defined and $\text{wt}(\mu,T)_i=\text{wt}(\mu,T)_{i+1}$ then $t_{s_i}.f_{\mu,T}$ is a generalized eigenvector of weight $\text{wt}(\mu,T)$ which is not a genuine eigenvector.  That (b) implies (c) is a straightforward check.

Finally, we prove that (c) implies (a).  Suppose that (c) holds and that $(\mu,S), (\nu,T) \in \Gamma$ with $\text{wt}(\mu,T)=\text{wt}(\nu,S)$.  Since $\text{wt}(\mu,S)=\text{wt}(\nu,T)$ for $1 \leq i \leq n$ we have
\begin{equation}
\mu_i-\nu_i=\beta(b)-\beta(b') \ \text{mod} \ r \quad \text{and} \quad \mu_i-\nu_i=d_{\beta(b)}-d_{\beta(b')}+r c_0 (\text{ct}(b)-\text{ct}(b')),
\end{equation} where $b=S(w_\mu(i))$ and $b'=T(v_\nu(i))$.  Thus $\mu_i-\nu_i=0$ mod $r$.  By (c), if $\mu_i \neq \nu_i$ then $b$ and $b'$ appear in the same component $\lambda^j$ of $\lambda$, $\lambda^j$ is either a single row or single column, and 
\begin{equation}
\text{ct}(b)>\text{ct}(b') \ \hbox{if $\lambda^j$ is a column, and} \quad \text{ct}(b)<\text{ct}(b') \ \hbox{if $\lambda^j$ is a row.}
\end{equation}  It follows that if $\mu_i>\nu_i$ for some $1 \leq i \leq n$ then
\begin{equation}
w_\mu(i)=S^{-1}(b)<S^{-1}(b')=S^{-1} T (v_\nu(i)).
\end{equation}  In particular, taking $i=w_\mu^{-1}(n)$ shows
\begin{equation}
\mu^-_n = \mu_{w_\mu^{-1}(n)} \leq \nu_{w_\mu^{-1}(n)}.
\end{equation}  By symmetry, $\nu_n^- \leq \mu_{v_\nu^{-1}(n)}$ and it follows that $\mu_n^-=\nu_n^-$ and $n=S^{-1} T v_\nu w_\mu^{-1} (n)$.  Now taking $i=w_\mu^{-1}(n-1)$ shows that if $\mu^-_{n-1}>\nu_{w_\mu^{-1}(n-1)}$ then
\begin{equation}
n-1<S^{-1} T v_\nu w_\mu^{-1}(n-1) \quad \implies \quad S^{-1} T v_\nu w_\mu^{-1}(n-1)=n,
\end{equation} contradicting our previous calculation.  We conclude as above that $\mu^-_{n-1}=\nu^-_{n-1}$ and $n-1=S^{-1} T v_\nu w_\mu^{-1}(n-1)$.  Continuing in this fashion we obtain $S^{-1} T v_\nu w_\mu^{-1}(i)=i$ for $1 \leq i \leq n$ which combined with the equation $\text{wt}(\mu,S)=\text{wt}(\nu,T)$ implies $(\mu,S)=(\nu,T)$, contradiction.
\end{proof}

If condition (a) of Lemma~\ref{simple spectrum} holds we say that the spectrum of $M(\lambda)$ is \emph{simple}.  When the spectrum of $M(\lambda)$ is simple, the \emph{calibration graph} of $M(\lambda)$ is the directed graph $\Gamma$ with vertex set 
\begin{equation} 
\Gamma=\{(\mu,T) \ | \ \mu \in \ZZ_{\geq 0}^n, T \in \text{SYT}(\lambda) \}
\end{equation} and with directed edges given by:
\begin{equation} \label{gamone}
(\mu,T) \rightarrow (s_i.\mu,T) \quad \iff \quad \mu_i \neq \mu_{i+1} \ \text{and} \ \sigma_i.f_{\mu,T} \neq 0,
\end{equation} and with $j=w_\mu(i)$
\begin{equation}
(\mu,T) \rightarrow (\mu,s_{j-1}.T) \quad \iff \quad \hbox{$\mu_i=\mu_{i+1}$ and $s_{j-1}.T$ is a standard tableau,}
\end{equation}
\begin{equation}
(\mu,T) \rightarrow (\phi.\mu,T) \quad \hbox{for all $\mu,T$,}
\end{equation} and when $\mu_n>0$
\begin{equation} \label{gamtwo}
(\mu,T) \rightarrow (\psi.\mu,T) \quad \iff \quad \Psi.f_{\mu,T} \neq 0.
\end{equation}  We also define the \emph{generic} calibration graph $\Gamma^{\text{gen}}$ by removing the conditions in \eqref{gamone} and \eqref{gamtwo}.  A subset $X \subseteq \Gamma$ is \emph{closed} if $(\mu,T) \in X$ and $(\mu,T) \rightarrow (\nu,S)$ implies that $(\nu,S) \in X$. 

\begin{lemma}
Suppose the $\ttt$-spectrum of $M(\lambda)$ is simple.  Then the set of submodules of $M(\lambda)$ is in bijection with the set of closed subsets of $\Gamma$, via the mapping associating to a submodule $M$ the set of $(\mu,T)$ with $f_{\mu,T} \in M$.  
\end{lemma}
\begin{proof}
Straightforward; the points are that every non-zero submodule $M$ of $M(\lambda)$ contains some $f_{\mu,T}$ and if a subspace $M$ is closed under $\ttt$, $\Phi,\Psi$, and $\sigma_1,\dots,\sigma_{n-1}$ then it is closed under all of $\HH$.
\end{proof}

Thus in the situation of Lemma \ref{simple spectrum}, the study of the submodule structure of $M(\lambda)$ is reduced to the study of the digraph $\Gamma$.  We next describe the closed subsets that can arise.  For a box $b \in \lambda$ and an integer $k \in \ZZ_{>0}$, define the subset $\Gamma_{b,k} \subseteq \Gamma$ by
\begin{equation}
\Gamma_{b,k}=\{(\mu,T) \in \Gamma \ | \ \mu^-_{T^{-1}(b)} \geq k \}
\end{equation}  where $\mu^-=w_\mu.\mu$ is the non-decreasing (i.e., anti-partition) rearrangement of $\mu$.  For an ordered pair of distinct boxes $b_1,b_2 \in \Gamma$ and an integer $k \in \ZZ_{>0}$, define the subset $\Gamma_{b_1,b_2,k}$ of $\Gamma$ by
\begin{align*}
(\mu,T) \in \Gamma_{b_1,b_2,k} \quad \iff \quad &\hbox{either $ \mu^-_{T^{-1}(b_1)} - \mu^-_{T^{-1}(b_2)} > k$} \\ 
& \hbox{or  $\mu^-_{T^{-1}(b_1)} - \mu^-_{T^{-1}(b_2)}=k$ and $w_\mu^{-1}(T^{-1}(b_1)) < w_\mu^{-1}(T^{-1}(b_2))$}.
\end{align*}  We write
\begin{equation} \label{fundamental submodules}
M_{b,k}=\CC \text{-span} \{f_{\mu,T} \ | \ (\mu,T) \in \Gamma_{b,k} \} \quad \text{and} \quad M_{b_1,b_2,k}=\CC \text{-span} \{ f_{\mu,T} \ | \ (\mu,T) \in \Gamma_{b_1,b_2,k} \}
\end{equation}  for the $\CC$-spans of $ \Gamma_{b,k}$ and $\Gamma_{b_1,b_2,k}$.

For an element $(\mu,T) \in \Gamma$ define the \emph{inversion set} $R(\mu,T)$ by
\begin{equation}
R(\mu,T)=\{ \Gamma_{b,k} \ | \ (\mu,T) \in \Gamma_{b,k} \} \cup \{\Gamma_{b_1,b_2,k} \ | \ (\mu,T) \in \Gamma_{b_1,b_2,k} \}.
\end{equation}  The size of $R(\mu,T)$ is a measure of how far $\mu$ is from the zero sequence.

The following technical lemma describes the properties of inversion sets we will need for our proof of Theorem \ref{submodules theorem}.  Its proof is a straightforward unwinding of the definitions and we omit it.

\begin{lemma} \label{inversion lemma}
\begin{enumerate}
\item[(a)]  For $(\mu,T) \in \Gamma$ put $k=\mu_1+1$ and $b=T(v_{\phi.\mu}(n))$.  Then
\begin{equation*}
R(\phi.\mu,T)=R(\mu,T) \cup \{\Gamma_{b,k} \}.
\end{equation*}

\item[(b)]  Suppose $\mu_i<\mu_{i+1}$ and put $k=\mu_{i+1}-\mu_i$, $b_1=T(w_\mu(i+1))$, and $b_2=T(w_\mu(i))$.  Then
\begin{equation*}
R(s_i.\mu,T)=R(\mu,T) \cup \{ \Gamma_{b_1,b_2,k} \}
\end{equation*}

\item[(c)]  Suppose $\mu_i=\mu_{i+1}$ and let $j=w_\mu(i)$.  Then
\begin{equation*}
R(\mu,s_{j-1}.T)=R(\mu,T).
\end{equation*}
\end{enumerate}

\end{lemma}

The next lemma describes the minimal length paths between two elements of $\Gamma^{\text{gen}}$.  It is the final combinatorial fact we need for the proof of Theorem \ref{submodules theorem}.

\begin{lemma} \label{paths lemma}
Let $(\mu,T), (\nu,S) \in \Gamma$ be two elements of $\Gamma$.  Then there is a sequence $(\mu,T)=(\mu_0,T_0),(\mu_1,T_1),\dots,(\mu_m,T_m)=(\nu,S)$ of elements of $\Gamma$ such that $(\mu_i,T_i)$ is adjacent to $(\mu_{i-1},T_{i-1})$ in $\Gamma^{\text{gen}}$ and $R(\mu_i,T_i)$ is either equal to $R(\mu_{i-1},T_{i-1})$ or is obtained from it by adjoining some element of $R(\nu,S)$ or by deleting some element not in $R(\nu,S)$.  
\end{lemma}
\begin{proof}
Define the distance between $(\mu,T)$ and $(\nu,S)$ by
\begin{equation}
d((\mu,T),(\nu,S))=|R(\mu,T) \Delta R(\nu,S)|+l(S^{-1} T)
\end{equation} where $X \Delta Y$ is the symmetric difference of the sets $X$ and $Y$ and $l(S^{-1}T)$ is the length of the permutation $S^{-1}T$.  

First suppose that 
\begin{equation} \label{badone}
(\nu,S) \notin \Gamma_{T(v_{\phi.\mu}(n)),\mu_1+1}, \quad (\nu,S) \in \Gamma_{T(w_\mu(n)),\mu_n},
\end{equation} and for all $1 \leq i \leq n-1$ 
\begin{equation} \label{badtwo}
\hbox{if $\mu_i<\mu_{i+1}$ then $(\nu,S) \notin \Gamma_{T(w_\mu(i+1)),T(w_\mu(i)),\mu_{i+1}-\mu_i}$},
\end{equation}
\begin{equation} \label{badthree}
\hbox{if $\mu_i>\mu_{i+1}$ then $(\nu,S) \in \Gamma_{T(w_\mu(i)),T(w_\mu(i+1)),\mu_i-\mu_{i+1}}$}.
\end{equation} and
\begin{equation} \label{badfour}
\hbox{if $\mu_i=\mu_{i+1}$ and $j=w_\mu(i)$ then either $s_{j-1}.T$ is not a tableau or $l(S^{-1} T s_{j-1}) > l(S^{-1} T)$.}
\end{equation}  We will show that in this case $(\mu,T)=(\nu,S)$.  First observe that by \eqref{badone}
\begin{equation}
\nu^-_{S^{-1}(T(w_\mu(1)))} \leq \mu_1 \quad \text{and} \quad \nu^-_{S^{-1}(T(w_\mu(n)))} \geq \mu_n
\end{equation}  If $\mu_1<\mu_2$ then by \eqref{badtwo}
\begin{equation}
\nu^-_{S^{-1}(T(w_\mu(2)))}-\nu^-_{S^{-1}(T(w_\mu(1)))} \leq \mu_2-\mu_1 \quad \implies \quad \nu^-_{S^{-1}(T(w_\mu(2)))} \leq \mu_2,
\end{equation}  with equality only if 
\begin{equation*}
\nu^-_{S^{-1}(T(w_\mu(1)))}=\mu_1 \quad \text{and} \quad v_\nu^{-1} S^{-1} T w_\mu(2) > v_\nu^{-1} S^{-1} T w_\mu(1);
\end{equation*}  if $\mu_1>\mu_2$ then by \eqref{badthree}
\begin{equation}
\nu^-_{S^{-1}(T(w_\mu(1)))}-\nu^-_{S^{-1}(T(w_\mu(2)))} \geq \mu_1-\mu_2 \quad \implies \quad \nu^-_{S^{-1}(T(w_\mu(2)))} \leq \mu_2,
\end{equation} with equality only if
\begin{equation*}
\nu^-_{S^{-1}(T(w_\mu(1)))}=\mu_1 \quad \text{and} \quad v_\nu^{-1} S^{-1} T w_\mu(1) <  v_\nu^{-1} S^{-1} T w_\mu(2);
\end{equation*} if $\mu_1=\mu_2$ then since $w_\mu(2)$ and $w_\mu(1)$ appear in adjacent boxes of $\lambda$ if $s_{j-1}.T$ is not a tableau, \eqref{badfour} implies 
\begin{equation*}
S^{-1} T(w_\mu(2))<S^{-1} T(w_\mu(1))
\end{equation*} and hence
\begin{equation}
\nu^-_{S^{-1}(T(w_\mu(2)))} \leq \nu^-_{S^{-1}(T(w_\mu(1)))} \leq \mu_1=\mu_2,
\end{equation}  with equality only if
\begin{equation*}
\nu^-_{S^{-1}(T(w_\mu(1)))}=\mu_1 \quad \text{and} \quad v_\nu^{-1} S^{-1} T w_\mu(2) > v_\nu^{-1} S^{-1} T w_\mu(1).
\end{equation*}  Continuing in this way we obtain
\begin{equation}
\nu^-_{S^{-1}(T(w_\mu(i)))} \leq \mu_i \quad \hbox{for $2 \leq i \leq n$},
\end{equation} with equality implying 
\begin{equation*}
\nu^-_{S^{-1}(T(w_\mu(i-1)))}=\mu_{i-1} \ \text{and} \  v_\nu^{-1} S^{-1} T w_\mu(i) > v_\nu^{-1} S^{-1} T w_\mu(i-1).
\end{equation*}  Since $\nu^-_{S^{-1}(T(w_\mu(n)))} \geq \mu_n$, the equalities all hold and hence
\begin{equation*}
\nu^-_{S^{-1}(T(w_\mu(i)))}=\mu_i \ \hbox{for $1 \leq i \leq n$ and} \ v_\nu^{-1} S^{-1} T w_\mu=1.
\end{equation*}  It follows that $\mu=\nu$ and $S=T$.

Now if at least one of the conditions \eqref{badone}, \eqref{badtwo}, \eqref{badthree}, and \eqref{badfour} does not hold then we can find $(\mu',T')$ adjacent to $(\mu,T)$ in $\Gamma^{\text{gen}}$ such that $d((\mu',T'),(\nu,S))<d((\mu,T),(\nu,S))$ and $R(\mu',T')$ is either equal to $R(\mu,T)$ or is obtained from it by adjoining an element of $R(\nu,S)$ or deleting an element not in $R(\nu,S)$, and the proof of the lemma is completed by induction.
\end{proof}

Finally we give our description of the set of submodules of $M(\lambda)$.

\begin{theorem} \label{submodules theorem}
Suppose that the $\ttt$-spectrum of $M(\lambda)$ is simple (see Lemma \ref{simple spectrum}), and recall the sets $M_{b,k}$ and $M_{b_1,b_2,k}$ defined in \eqref{fundamental submodules}.  
\begin{enumerate}
\item[(a)]  Let $k \in \ZZ_{>0}$ and $b \in \lambda^i$. Then $M_{b,k}$ is an $\HH$-submodule of $M(\lambda)$ exactly if $k=d_i-d_{i-k}+r \text{ct}(b) c_0$.
\item[(b)]  Let $k \in \ZZ_{>0}$ and $b_1,b_2 \in \lambda$.  Then $M_{b_1,b_2,k}$ is an $\HH$-submodule of $M(\lambda)$ exactly if $k=\beta(b_1)-\beta(b_2)$ mod $r$ and $k=d_{\beta(b_1)}-d_{\beta(b_2)}+r (\text{ct}(b_1)-\text{ct}(b_2) \pm 1) c_0$.
\item[(c)] Every submodule of $M(\lambda)$ is in the lattice generated by those of types (a) and (b).
\end{enumerate}
\end{theorem} 
\begin{proof}
For (a), we must check when $\Gamma_{b,k}$ is a closed subset of $\Gamma$. By Lemma \ref{inversion lemma} we need only establish the conditions under which: if $\mu_n \neq 0$, $(\mu,T) \in \Gamma_{b,k}$, and $(\psi.\mu,T) \notin \Gamma_{b,k}$ then
\begin{equation} \label{check1}
\Psi.f_{\mu,T}=0.
\end{equation}  Since $(\mu,T) \in \Gamma_{b,k}$ but $(\psi.\mu,T) \notin \Gamma_{b,k}$, we must have $\mu_n=k$ and there are precisely $n-T^{-1}(b)+1$ parts of $\mu$ of size at least $k$.  Hence
\begin{equation}
w_\mu(n)=n-|\{1 \leq i \leq n \ | \ \mu_i \geq k \}|+1=T^{-1}(b).
\end{equation}  By part (e) of Lemma \ref{action lemma},
\begin{align*}
\Psi.f_{\mu,T}&=\left(\mu_n-(d_{\beta(T(w_\mu(n)))}-d_{\beta(T(w_\mu(n)))-\mu_n})-r c_0 \text{ct}(T(w_\mu(n)) \right) f_{\psi.\mu,T} \\
&=\left(k-(d_i-d_{i-k})-r c_0 \text{ct}(b) \right) f_{\psi.\mu,T}
\end{align*} by our assumption, proving (a).

The proof of (b) is similar. By using Lemma \ref{inversion lemma} we need only establish the conditions under which: if $(\mu,T) \in \Gamma_{b_1,b_2,k}$ and $(s_i.\mu,T) \notin \Gamma_{b_1,b_2,k}$, then $\sigma_i.f_{\mu,T}=0$.  By definition of $\Gamma_{b_1,b_2,k}$ we have
\begin{equation*}
w_\mu^{-1} T^{-1}(b_1)=i, \ w_\mu^{-1} T^{-1}(b_2)=i+1, \quad \text{and} \quad \mu^-_{T^{-1}(b_1)}-\mu^-_{T^{-1}(b_2)}=k,
\end{equation*}  and hence
\begin{equation}
\mu_i-\mu_{i+1}=\mu^-_{T^{-1}(b_1)}-\mu^-_{T^{-1}(b_2)}=k=\beta(b_1)-\beta(b_2) \ \text{mod} \ r.
\end{equation}
By Lemma \ref{action lemma} part (b)
\begin{align*}
\sigma_i.f_{\mu,T}&=\left(k-(d_{\beta(b_1)}-d_{\beta(b_2)})-r c_0(\text{ct}(b_1)-\text{ct}(b_2) +1) \right) \\
&\times \left(k-(d_{\beta(b_1)}-d_{\beta(b_2)})-r c_0(\text{ct}(b_1)-\text{ct}(b_2) -1) \right) f_{s_i.\mu,T}.
\end{align*}  This proves (b).

We now prove (c).  Fix $(\mu,T) \in \Gamma$.  By using Lemma \ref{paths lemma} the submodule generated by $f_{\mu,T}$ is equal to 
\begin{align*}
\CC \text{-span} &\{f_{\nu,S} \ | \ R(\mu,T) \cap C \subseteq R(\nu,S) \cap C \} \\
&=\bigcap_{\substack{\Gamma_{b,k} \text{closed} \\ (\mu,T) \in \Gamma_{b,k}}} M_{b,k} \cap \bigcap_{\substack{\Gamma_{b_1,b_2,k} \text{closed} \\ (\mu,T) \in \Gamma_{b_1,b_2,k}}} M_{b_1,b_2,k}
\end{align*} where
\begin{equation}
C= \{ \Gamma_{b,k} \ | \ \Gamma_{b,k} \ \text{is closed} \} \cup \{\Gamma_{b_1,b_2,k} \ | \ \Gamma_{b_1,b_2,k} \ \text{is closed} \}.
\end{equation}  Here we used the fact that if $\sigma_i.f_{\mu,T}=0$ then $(\mu,T) \in \Gamma_{b_1,b_2,k}$ and $(s_i.\mu,T) \notin \Gamma_{b_1,b_2,k}$ for some closed set $\Gamma_{b_1,b_2,k}$, and if $\mu_n>0$ and $\Psi.f_{\mu,T}=0$ then $(\mu,T) \in \Gamma_{b,k}$ and $(\psi.\mu,T) \notin \Gamma_{b,k}$ for some closed set $\Gamma_{b,k}$.

Since any submodule is equal to the span of the eigenvectors $f_{\mu,T}$ that it contains, (c) is proved. 
\end{proof}

\begin{remark}
The paper \cite{Gor} uses the representation theory of rational Cherednik algebras to prove a conjecture of Haiman on the ring of diagonal coinvariants of a Weyl group $W$.  The key point is the calculation of the graded $W$-character of $L(\one)$, the irreducible $\HH$-module corresponding to the trivial $W$-module.  In \cite{Val}, an analog of this theorem is proved for the groups $G(r,p,n)$ when $p<r$; a simplified version of Theorem \ref{submodules theorem} is used in 
\cite{Gri2} to obtain the results of \cite{Val} without restriction on $p$ and without appealing to the KZ functor.
\end{remark}

\section{Examples and some finite dimensional $\HH$-modules}

We next illustrate the scope of the theorem in the cases $W=G(1,1,n)$ and $W=G(2,1,n)$, and close by showing how it may be applied to produce a large supply of finite dimensional $\HH$-modules when $r$ is large.

\subsection{$G(1,1,n)$} \label{typeA}

In case $W=G(1,1,n)$, and assuming for simplicity that $\lambda$ is not a row or column, the scope of our results is quite limited: $\ttt$ acts with simple spectrum on $M(\lambda)$ (for $\lambda$ a partition) exactly if
\begin{equation*}
c_0 \notin \bigcup_{l=1}^d \frac{1}{l} \ZZ
\end{equation*} where $d=\text{ct}^+(\lambda)-\text{ct}^-(\lambda)$ is the difference between the largest and smallest contents of boxes of $\lambda$ (it does not seem unreasonable to call $d$ the ``diameter'' of $\lambda$).  In this case, $M(\lambda)$ has a proper non-zero submodule exactly if
\begin{equation}
c_0=\frac{m}{d+1} \quad \hbox{for some integer $m$ with $(m,d)=1$.}
\end{equation}  In fact, assuming $m>0$ and writing $b_1$ (respectively, $b_2$) for the upper right-hand (respectively, lower left-hand) box of $\lambda$, part (c) of Theorem \ref{submodules theorem} shows that the unique proper non-zero submodule of $M(\lambda)$ is $M_{b_1,b_2,m}$.  In this case it should be possible to obtain a concordance with the results of \cite{Suz}, where the $\ttt$-diagonalizable $L(\lambda)$'s are analyzed.  None of these is finite dimensional.  By the results of \cite{BEG} there are no finite dimensional modules for the rational Cherednik algebra of type $G(1,1,n)$, and the quotient of $L(\lambda)$ by $x_1+x_2+\cdots+x_n$ is finite dimensional exactly if $\lambda=(n)$ and $c_0=k/n$ with $(k,n)=1$ and $k \in \ZZ_{>0}$.

\subsection{$G(2,1,n)$} \label{typeB}
If $W=G(2,1,n)$ is the Weyl group of type $B_n$, slightly more interesting things can happen.  Assume that $\lambda$ is a partition of $n$ that is not a single row or column.  Then as in the previous example the $\ttt$-spectrum of $M(\lambda,\emptyset)$ is simple exactly if 
\begin{equation}
c_0 \notin \bigcup_{l=1}^d  \frac{1}{l} \ZZ.
\end{equation}  For a box $b \in \lambda$, the set $M_{b,k}$ is a submodule if and only if $k$ is a positive odd integer and
\begin{equation}
d_0-d_1+2\text{ct}(b)c_0=k;
\end{equation} given the assumption on the $\ttt$-spectrum of $M(\lambda)$ at most one such equation can hold, but it may give rise to multiple submodules $M_{b,k}$ (one for each box $b$ with the appropriate content).  On the other hand, if $b_1$ and $b_2$ are the upper right-hand and lower left-hand corners of $\lambda$, then $M_{b_1,b_2,k}$ is a submodule exactly if $k$ is an even integer and
\begin{equation}
2 (d+1)c_0=k.
\end{equation}  There is only one submodule corresponding to this equation.  Some examples of finite dimensional modules $L(\lambda,\emptyset)$ follow (proofs in a more general setting are provided by Corollary~\ref{finite dimensional}): 

\begin{itemize}

\item[(a)]  $\lambda$ is a rectangle, in which case $L(\lambda,\emptyset)$ is finite dimensional exactly if
\begin{equation*}
d_0-d_1+2 c c_0=k \quad \hbox{or by using the relation $d_0+d_1=0$,} \quad d_0+c_0=k/2
\end{equation*} for some odd $k \in \ZZ_{>0}$, where $c$ is the content of the lower right-hand corner of $\lambda$.  These representations are a special case of the representations constructed for wreath product symplectic reflection algebras in \cite{EtMo}.  Their dimension is given by
\begin{equation*}
\text{dim}_\CC(L(\lambda))=k^n \text{dim}(S^\lambda)
\end{equation*} 

\item[(b)] $\lambda$ is of the form $\lambda=(m_1+m_2,m_1,m_1,\dots,m_1)$ for some positive integers $m_1,m_2$ and the parameters satisfy
\begin{equation*}
d_0-d_1+2 c c_0=k \quad \text{and} \quad 2 (d+1) c_0 =l 
\end{equation*}  for some positive odd integer $k$ and positive even integer $l$, where $c$ is the content of the lower outside corner box of $\lambda$ and (as in \ref{typeA}) $d$ is the diameter of $\lambda$.  A combinatorial formula for the dimension of $L(\lambda)$ follows from Theorem \ref{submodules theorem}.  We hope to return to the problem of giving uniform combinatorial formulas for the dimensions (and graded $W$-characters) of the finite dimensional $\HH$-modules once we have understood the submodule structure in greater generality.

\item[(c)] $\lambda$ is the transpose of type (b).  We leave the details here to the reader.

\end{itemize}

As far as we are aware, types (b) and (c) have not appeared before.  They do not show up if the parameters $c_1=d_0$ and $c_0$ are equal, hence are (so far) invisible from the ``geometric'' point of view.

The case when $\lambda$ consists of two non-empty partitions is in some sense simpler and we leave it to the interested reader to work out the possible submodules.

\subsection{Finite dimensional $\HH$-modules} \label{typer}
An \emph{outside corner} of an $r$-partition $\lambda$ is a box $b$ of $\lambda$ such that there exists a standard tableau $T \in \text{SYT}(\lambda)$ with $T^{-1}(b)=n$.  After reading the statement of the next corollary, the reader may wish to refer to Theorem~\ref{submodules theorem} for the equational characterization of when certain subsets of $M(\lambda)$ are submodules. 
\begin{corollary} \label{finite dimensional}
Suppose that $\lambda$ is an $r$-partition of $n$ and that the spectrum of $M(\lambda)$ is simple.  If for every outside corner $b$ of $\lambda$ there exists a sequence $b=b_0,b_1,\dots,b_{2m-1},b_{2m}$ of boxes of $\lambda$ and positive integers $k_0,k_1,\dots,k_m$ such that
\begin{enumerate}
\item[(a)] for $1 \leq i \leq m$, we have $\beta(b_{2i})=\beta(b_{2i-1})$ and $b_{2i}$ appears to the right and below $b_{2i-1}$, 
\item[(b)] for $0 \leq i \leq m-1$ the set $M_{b_{2i},b_{2i+1},k_i}$ is an $\HH$-submodule, and
\item[(c)] the set $M_{b_{2m},k_m}$ is an $\HH$-submodule,
\end{enumerate}
then $L(\lambda)$ is finite dimensional. 
\end{corollary}
\begin{proof}
By Theorem~\ref{submodules theorem}, a basis for $L(\lambda)$ is indexed by
\begin{equation} \label{basis set}
\{(\mu,T) \ | \ (\mu,T) \notin \Gamma_{b,k} \ \hbox{if $\Gamma_{b,k}$ is closed and} \ (\mu,T) \notin \Gamma_{b_1,b_2,k} \ \hbox{if $\Gamma_{b_1,b_2,k}$ is closed} \}.
\end{equation}  Fix $T \in \text{SYT}(\lambda)$.  Let $b=T(n)$ so that $b$ is an outside corner of $\lambda$.  Note that for any boxes $b,b'$ of $\lambda$ and an integer $k$, if $(\mu,T) \notin \Gamma_{b,k}$ then $\mu^-_{T^{-1}(b)} < k$ and if $(\mu,T) \notin \Gamma_{b,b',k}$ then $\mu^-_{T^{-1}(b)} \leq \mu^-_{T^{-1}(b')}+k$.  Furthermore if $b'$ appears to the right and below $b$ then $T^{-1}(b') \geq T^{-1}(b)$.  Therefore our hypotheses imply that if $(\mu,T)$ is an element of the set \eqref{basis set} then $\mu^-_n < k_0+k_1+\cdots+k_m$, and hence the set \eqref{basis set} is finite.
\end{proof}  

As a particular special case, let $b_1,b_2,\dots,b_m \in \lambda$ be the outside corner boxes of $\lambda$.  If the spectrum of $M(\lambda)$ is simple and there are positive integers $k_1,k_2,\dots,k_m$ with
\begin{equation}
d_{\beta(b_i)}-d_{\beta(b_i)-k_i}+r c_0 \text{ct}(b_i)=k_i \quad \hbox{for $1 \leq i \leq m$}
\end{equation} then Corollary \ref{finite dimensional} shows that $L(\lambda)$ is finite dimensional.  The assumption that the spectrum of $M(\lambda)$ is simple may mean no such integers exist, but if $r \geq 2n$ then a dimension count shows that for every $r$-partition $\lambda$ of $n$  there is a choice of parameters for which the $\ttt$-spectrum of $M(\lambda)$ is simple and $L(\lambda)$ is finite dimensional.  More generally, if $N$ is the number of non-empty $\lambda^i$'s and the number $m$ of corner boxes of $\lambda$ satisfies $m \leq r-N$, then there is a choice of parameters for which $L(\lambda)$ is finite dimensional.  Certain of these modules (those in which each $\lambda^i$ is a rectangle) have been considered by \cite{Gan} and \cite{Mon} in the broader context of wreath product symplectic reflection algebras.

We now give a somewhat more involved special case in order to illustrate the extent of Corollary~\ref{finite dimensional}.  Let $\lambda$ be the $5$-partition
\begin{equation*}
\left(\begin{array}{cccc} \begin{array}{c} \young(\hfil\hfil\bone,\hfil\btwo) \end{array}, &  \begin{array}{@{}c} \young(\bthree\hfil\bfour) \end{array}, &  \begin{array}{@{}c} \emptyset \end{array}, &  \begin{array}{@{}c} \young(\hfil\hfil\hfil\hfil\bseven,\hfil\hfil\hfil\hfil,\bfive\hfil\hfil\bsix) \end{array},  \begin{array}{@{}c} \emptyset \end{array}
\end{array} \right),
\end{equation*} with boxes labeled $b_1,b_2,b_3,b_4,b_5,b_6$ and $b_7$ as shown.  Suppose the parameters satisfy the system of equations
\begin{align*}
&d_{\beta(b_2)}-d_{\beta(b_2)-3}+5 \text{ct}(b_2) c_0=3  \\
&d_{\beta(b_6)}-d_{\beta(b_6)-4}+5 \text{ct}(b_6) c_0=4 \\
&d_{\beta(b_1)}-d_{\beta(b_3)}+5(\text{ct}(b_1)-\text{ct}(b_3)+1) c_0=9  \\
&d_{\beta(b_4)}-d_{\beta(b_5)}+5( \text{ct}(b_4)-\text{ct}(b_5)+1) c_0=3 \\
&d_{\beta(b_7)}-d_{\beta(b_5)}+5(\text{ct}(b_7)-\text{ct}(b_5)+1) c_0=10
\end{align*} and that the spectrum of $M(\lambda)$ is simple (in fact, the unique parameter choice satisfying our system is $c_0=2/7,d_1=-17/7,d_2=-5/7,d_3=12/7,d_4=-6/7$, and with this choice the spectrum is simple).  It is routine to check that for each of the corner boxes $b_1,b_2,b_4,b_5,$ and $b_7$ of $\lambda$, the conditions of Corollary~\ref{finite dimensional} hold: for the box $b_2$, this is ensured by the first equation, for $b_6$ by the second equation, for $b_4$ by the fourth and second equations, for $b_7$ by the fifth and second equations, and for $b_1$ by the third, fourth, and second equations.

\section{Clifford theory and the descent to $G(r,p,n)$} \label{Clifford section}

In this section we write $\HH$ for the rational Cherednik algebra of type $G(r,1,n)$, we fix a positive integer $p$ dividing $r$, and we assume that $d_i=d_j$ for $i=j$ mod $r/p$.  We write $\HH_p$ for the rational Cherednik algebra of type $G(r,p,n)$.  If we also assume that $n \geq 3$ then $\HH_p$ may be realized as the subalgebra of $\HH$ fixed by the cyclic group of automorphisms generated by the map $\alpha$ given by
\begin{equation}
\alpha(x_i)=x_i, \quad \alpha(y_i)=y_i, \quad \alpha(t_w) = t_{w}, \quad \text{and} \quad \alpha(t_{\zeta_i})=\zeta^{r/p} t_{\zeta_i}
\end{equation} for $1 \leq i \leq n$ and $w \in G(1,1,n)$.  We will modify the version of Clifford theory given in the appendix of \cite{RaRa} to deduce information on the Verma modules for $\HH_p$ from information on the Verma modules for $\HH$.  The results of this section can be used to relate representations of $\HH$ to those of $\HH_p$ even when the functions $f_{\mu,T}$ are not well-defined.  Since this technique is now standard, we do not include proofs of most of the assertions made here; the interested reader should have no difficulty supplying them.

Let $M$ be an $\HH$-module, and define the $\alpha$-\emph{twisted} module $M^\alpha$ to be the $\HH$-module whose underlying vector space is $M$, but with $\HH$-action given by
\begin{equation}
f.m^\alpha=\alpha(f).m \quad \hbox{for $f \in \HH$ and $m \in M$,}
\end{equation} where $m^\alpha$ is the element of $M^\alpha$ corresponding to $m \in M$.  Notice that the map $m \mapsto m^\alpha$ is an $\HH_p$-module isomorphism and that the map $N \mapsto N^\alpha$ is an isomorphism of the lattice of $\HH$-submodules of $M$ onto the lattice of $\HH$-submodules of $M^\alpha$.  In particular, the radical of $M^\alpha$ is $\text{rad}(M)^\alpha$.

For the Verma modules $M(\lambda)$, we can make all this more explicit: let $C$ be the cyclic shift operator on $r$-partitions defined by $C.(\lambda_0,\lambda_1,\dots,\lambda_{r-1})=(\lambda_{r-1},\lambda_0,\dots,\lambda_{r-2})$ and also write $C$ for the bijection $C:\text{SYT}(\lambda) \rightarrow \text{SYT}(C.\lambda)$ given by cyclic shifting.  Define an $F$-linear map $C:M(\lambda) \rightarrow M(C.\lambda)$ by
\begin{equation}
C(f v_T)=f v_{C.T} \quad \hbox{for $f \in S(\hh^*)$ and $T \in \text{SYT}(\lambda)$.}
\end{equation}  Then the map $f v_{C^{r/p}.T} \mapsto (f v_T)^\alpha$ defines an isomorphism $M(C^{r/p}.\lambda) \cong M(\lambda)^{\alpha}$ of $\HH$-modules.  Now let $k$ be the smallest integer so that $M(\lambda)^{\alpha^k} \cong M(\lambda)$ as $\HH$-modules.  Then $k$ is also the smallest integer so that $C^{kr/p}.\lambda=\lambda$, and the map $C^{kr/p}$ is an $\HH_p$-module automorphism of $M(\lambda)$.  For $0 \leq q \leq p/k-1$ let $M(\lambda,q)$ be the the $\zeta^{qkr/p}$-eigenspace of $C^{kr/p}$ on $M(\lambda)$.  Thus
\begin{equation}
M(\lambda)=\bigoplus_{q=0}^{p/k-1} M(\lambda,q)
\end{equation} and we let $\pi_q:M(\lambda) \rightarrow M(\lambda,q)$ denote the projection onto the $q$th summand.  The formula $\pi_q(t_{\zeta_1}.f)=t_{\zeta_1}.\pi_{q+1}(f)$ show that $t_{\zeta_1}$ induces a vector space isomorphism of $M(\lambda,q+1)$ onto $M(\lambda,q)$ (where $q$ is to be read modulo $p/k$).

It is a standard part of the representation theory of $G(r,p,n)$ that the non-zero eigenspaces of $C^{kr/p}$ restricted to $S^\lambda \subseteq M(\lambda)$ are pairwise non-isomorphic irreducible $G(r,p,n)$-modules and that as $\lambda$ ranges over a set of representatives for the $C^{r/p}$-orbits on $r$-partitions of $n$ we obtain each irreducible $G(r,p,n)$-module exactly once.  It follows that the module $M(\lambda,q)$ is a Verma module for $\HH_p$, and that all the Verma modules for $\HH_p$ arise in this way as summands of Verma modules for $\HH$.  The following theorem relates the irreducible heads of the Verma modules for $\HH$ and $\HH_p$.  It is easily obtained from what we have done.  For a graded vector space $A$, we write $A_i$ for the $i$th graded piece.

\begin{theorem}
The radical of $M(\lambda,q)$ is $\pi_q(\text{rad}(M(\lambda))$.  Furthermore, $\text{rad}(M(\lambda))$ is a $C^{kr/p}$-stable submodule of $M(\lambda)$, so $C^{kr/p}$ acts on $L(\lambda)$ with eigenspace decomposition 
\begin{equation*}
L(\lambda)=\bigoplus_{0 \leq q \leq p/k-1} L(\lambda,q),
\end{equation*} where $L(\lambda,q)$ is the irreducible head of $M(\lambda,q)$.  The element $t_{\zeta_1} \in \HH$ maps $M(\lambda,q)$ into $M(\lambda,q-1)$ and induces a vector space isomorphism of $L(\lambda,q)$ onto $L(\lambda,q-1)$ ($q$ and $q-1$ should be taken modulo $p/k$); in particular the graded dimension of $L(\lambda,q)$ is given by $\text{dim}(L(\lambda,q)_i=\frac{k}{p} \text{dim}(L(\lambda)_i)$.
\end{theorem}

In particular, in those cases in which the $\ttt$-spectrum of $M(\lambda)$ is simple, the preceding theorem combined with the results of Section~\ref{submodules section} can be used to obtain explicit bases for the modules $L(\lambda,q)$.

{\bf Acknowledgements} I thank Peter Webb and Victor Reiner for many interesting discussions during the time this paper was being written, and Arun Ram for teaching me about intertwining operators and for directing me to some valuable references.  Finally, I thank Charles Dunkl for patiently explaining some of his recent work on singular polynomials for the symmetric group.  His paper \cite{DuOp} with E. Opdam provided the initial inspiration for the present work, and Theorem \ref{submodules theorem} is a first attempt to answer some of his questions.  Partial support was provided by NSF grant DMS-0449102.

\def\cprime{$'$} \def\cprime{$'$}

\end{document}